\documentclass{amsart}

\usepackage{amsmath,amsthm,amssymb,amsfonts}
\usepackage{graphicx,color}

\newcommand{\R}{\mathbb{R}}

\newcommand{\eps}{\varepsilon}

\def\dx{\,{\rm d}x}
\def\dy{\,{\rm d}y}
\def\dt{\,{\rm d}t}
\def\dta{\,{\rm d}\tau}
\def\ds{\,{\rm d}s}

\def\e{\varepsilon}
\def\vphi{\varphi}

\newtheorem{theorem}{Theorem}
\newtheorem{proposition}[theorem]{Proposition}
\newtheorem{lemma}[theorem]{Lemma}
\newtheorem{corollary}[theorem]{Corollary}

\theoremstyle{definition}
\newtheorem{definition}[theorem]{Definition}

\theoremstyle{remark}
\newtheorem{remark}[theorem]{Remark}

\DeclareMathOperator{\sgn}{sgn}

\numberwithin{equation}{section}
\numberwithin{theorem}{section}

\author{Piotr Biler}
\address{P. Biler:
 Instytut Matematyczny, Uniwersytet Wroc\l awski,
 pl. Grunwaldzki 2/4, 50-384 Wroc\-\l aw, Poland}
\email{Piotr.Biler@math.uni.wroc.pl}

\author{Cyril Imbert}
\address{C. Imbert: CNRS, UMR 8050, 
Universit\'e   Paris-Est Cr\'eteil, 
61 av. du G\'en\'eral de Gaulle, 94010 Cr\'eteil, cedex,  France}
\email{cyril.imbert@u-pec.fr} 
\urladdr{http://www.perso-math.univ-mlv.fr/users/imbert.cyril/}

\author{Grzegorz Karch}
\address{G. Karch:
 Instytut Matematyczny, Uniwersytet Wroc\l awski,
 pl. Grunwaldzki 2/4, 50-384 Wroc\-\l aw, Poland}
\email{Grzegorz.Karch@math.uni.wroc.pl}
\urladdr{http://www.math.uni.wroc.pl/~karch}

\title[Nonlocal porous medium equation]{Nonlocal  porous medium equation: 
 Barenblatt  profiles and other weak solutions}

\begin{document}

\begin{abstract} 
A degenerate nonlinear nonlocal evolution equation is considered; it
can be understood as a porous medium equation whose pressure law is
nonlinear and nonlocal. We show the existence of sign changing weak
solutions to the corresponding Cauchy problem.  Moreover, we construct
explicit compactly supported self-similar solutions which generalize
Barenblatt profiles --- the well-known solutions of the classical
porous medium equation.
\end{abstract}

\keywords{porous medium equation, nonlocal equation, 
hypercontractivity, self-similar solutions}

\subjclass[2000]{35K55, 35B45, 35C06}

\date{\today}

\thanks{The authors wish to thank Jean Dolbeault and R\'egis Monneau for
  the fruitful discussions they had together.  The authors were supported
  by an EGIDE project (PHC POLONIUM 20078TL, 0185, 2009--2010).  The
  second author was supported by an ANR project (EVOL).  The first and
  the third authors were supported by the \hbox{MNSzW} grant N201
  418839, and the Foundation for Polish Science operated within the
  Innovative Economy Operational Programme 2007--2013 funded by
  European Regional Development Fund (Ph.D. Programme: Mathematical
  Methods in Natural Sciences). }

\maketitle

\baselineskip=20pt

\section{Introduction}

In this work, we study the following degenerate nonlinear nonlocal evolution equation
\begin{equation}\label{eq:main}
\partial_t u = \nabla \cdot \left( |u| \nabla^{\alpha-1} (|u|^{m-2}u)
\right),
\qquad x\in \R^d,\ t>0,
\end{equation}
where $m >1$ and $\nabla^{\alpha-1}$ denotes the integro-differential
operator $\nabla (-\Delta)^{\frac{\alpha}{2}-1}$, $\alpha\in (0,2)$.
The equation is supplemented with an initial condition
\begin{equation}\label{eq:ic}
u(0,x) = u_0 (x).
\end{equation}
First, we construct nonnegative self-similar solutions of equation
\eqref{eq:main} which are \emph{explicit} and \emph{compactly
  supported}. They generalize the classical
Barenblatt--Kompaneets--Pattle--Zel'\-do\-vich solutions of the porous
medium equation, see \eqref{PME} below.  Second, we prove the
existence of sign changing weak solutions to problem
\eqref{eq:main}--\eqref{eq:ic} for merely integrable initial data, and
we prove that these solutions satisfy sharp hypercontractivity
$L^1\mapsto L^p$ estimates.

\subsection*{A nonlocal operator.}
Equation~\eqref{eq:main} involves a nonlocal operator denoted by
$\nabla^{\alpha-1}$ which can be defined as the Fourier multiplier
whose symbol is ${i \xi}|\xi|^{\alpha-2}$. This notation emphasizes
that it is a (pseudo-differential) operator of order $\alpha -1$.
Recalling the definition of the fractional Laplace operator
$(-\Delta)^{\frac\alpha2}(v)={\mathcal F}^{-1}(|\xi|^\alpha{\mathcal
  F}v)$ and the Riesz potential ${\mathcal
  I}_\beta=(-\Delta)^{-\frac\beta2}$, {\it i.e.}  Fourier multipliers
whose symbols are $|\xi|^\alpha$ and $|\xi|^{-\beta}$ respectively (see
for instance \cite[Ch.~V]{Stein}), the fractional gradient
$\nabla^{\alpha-1}$ can also be written as $\nabla {\mathcal
  I}_{2-\alpha}$. Finally, let us emphasize that the definition of
$\nabla^{\alpha-1}$ is consistent with the usual gradient:
$\nabla^1=\nabla$; the components of $\nabla^0$ are the Riesz
transforms; moreover we have $\nabla\cdot\nabla^{\alpha-1}
=\nabla^{\frac\alpha2}\cdot\nabla^{\frac\alpha2}
=-(-\Delta)^{\frac\alpha2}$.  It is also possible, following the
reasoning from \cite[Th. 1]{DI06} to define the fractional gradient
via the singular integral formula for smooth and bounded functions
$v:\R^d\to \R$
\begin{equation}\label{frac:int}
\nabla^{\alpha-1} v (x) = C_{d,\alpha} \int \big(v(x)-v(x+z)\big) \frac{z}{|z|^{d+\alpha}} \,{\rm d}z 
\end{equation}
with a suitable constant $C_{d,\alpha}>0$.

\subsection*{Related equations and results.} Our preliminary results for the problem 
\eqref{eq:main}--\eqref{eq:ic} have been announced in \cite{BIK}.

First, we would like to shed light on the link between \eqref{eq:main}
and other partial differential equations.  Notice that when $\alpha=2$
equation \eqref{eq:main} coincides with the classical (nonlinear
parabolic) porous medium equation
\begin{equation}\label{PME}
\partial_t u = \nabla \cdot \big( |u|\nabla (|u|^{m-2}u) \big)\ \ =\nabla\cdot\big((m-1)|u|^{m-1}\nabla u\big).
\end{equation}
For the theory of porous media equations, the interested reader is
referred to \cite{V2,V1} and references therein. Of course, for $m=2$,
the Boussinesq equation is recovered.

The following nonlinear and nonlocal equation 
\begin{equation}\label{disloc}
\partial_tv+|v_x|\left(-\frac{\partial^2}{\partial x^2}\right)^\frac{\alpha}{2}v=0
\end{equation}
in the one-dimensional case $x\in\R$ was studied by the first, the
third authors and R. Monneau \cite{bkm}.  Such an equation was derived
as a model for the dynamics of dislocations in crystals. In
\cite{bkm}, the existence, uniqueness and comparison properties of
(viscosity) solutions have been proved, and explicit self-similar
solutions have been constructed.  Notice that the function $u=v_x$,
where $v$ is a solution to \eqref{disloc}, solves the one-dimensional
case of \eqref{eq:main} with $m=2$. Thus, equation \eqref{eq:main} is
a multidimensional generalization of the one in \eqref{disloc}.

Recently, Caffarelli and V\'azquez \cite{cv10a,cv10b} studied
nonnegative weak solutions of \eqref{eq:main} in the case $m=2$ in the
multidimensional case. Precisely, they studied the following
(nonlocal) porous medium equation in $\R^d$
\begin{equation}\label{eq:cv}
\partial_t u = \nabla \cdot (u \nabla p), 
\end{equation}
with the nonlocal pressure law $p = (-\Delta)^{-s} u$, $0<s<1$,
obtained from the density $u\ge 0$. Notice that, for $\alpha = 2 -2s
\in (0,2)$, equation \eqref{eq:cv} reads $\partial_t u = \nabla \cdot
( u \nabla^{\alpha-1} u)$.  For sign changing $u$'s, our
equation~\eqref{eq:main} is a (formally parabolic) extension of
equation~\eqref{eq:cv} of the structure of \eqref{disloc}.  In
\cite{cv10a}, Caffarelli and V\'azquez constructed nonnegative weak
solutions for \eqref{eq:cv}, {\it i.e.} for \eqref{eq:main} with
$m=2$, with initial data satisfying: $u_0 \in L^1(\R^d)\cap
L^\infty(\R^d)$ and such that $ 0 \le u_0 (x) \le A {\rm e}^{-a|x|}$
for some $ A,\, a >0$.  Besides the positivity and the mass
preservation, the properties of solutions, listed in the next paper
\cite[p. 4]{cv10b}, include the finite speed of propagation proved
using the comparison with suitable supersolutions.  Further regularity
properties of solutions of \eqref{eq:main} with $m=2$ and
$\alpha\in(0,1)$ are studied in \cite{CSV}.

Another nonlocal porous medium equation has been proposed in
\cite{dPQRV-AdvM,dPQRV-CPAM,JLV-2012} 
\[\partial_t u+(-\Delta)^{\frac{\alpha}{2}}(|u|^{m-1}u)=0\]
for $\alpha=1$ and $\alpha\in(0,2)$, respectively.  Among several other
properties like smoothing effects and decay estimates, solutions of
this generalization of the porous medium equation enjoy the
$L^1$-contraction property, so, they are unique. But self-similar
solutions are not compactly supported \cite[Th.~1.1]{JLV-2012}.

Finally, we recall that the following nonlocal higher order equation,
appearing in the modeling of propagation of fractures in rocks,
\[ \partial_tu=\nabla\cdot(u^n\nabla(-\Delta)^{\frac12}u)\]
(with $u\ge 0$ and $n >1$), has been studied in \cite{imbertmellet} in a
one-dimensional bounded domain. At least formally, this equation with
$n=1$ corresponds to \eqref{eq:main} with $\alpha=3$ and $m=2$.

\subsection*{Notation.}  

In this work, $Q_T$ denotes $(0,T) \times \R^d$. The usual norm of the
Lebesgue space $L^p (\R^d)$ is denoted by $\|\,\cdot\,\|_p$ for any $p
\in [1,\infty]$, and $H^{s,p}(\R^d)$ with the norm
$\|\cdot\|_{H^{s,p}}$ is the fractional order Sobolev space, see
Section~\ref{sec:prelim}.  The Fourier transform $\mathcal F$ and its 
inverse transform $\mathcal{F}^{-1}$ of a function $v\in
L^1(\R^d)$ are defined by
\[ \mathcal{F} v (\xi) = (2\pi)^{-\frac{d}2} \int v(x)\, {\rm e}^{-i x \cdot \xi} \,{\rm d}x, 
\ \ \ \mathcal{F}^{-1} v (x) = (2\pi)^{-\frac{d}2} \int v(x)\, {\rm
  e}^{i x \cdot \xi} \,{\rm d}\xi. \] 
Here, all integrals with no
integration limits are over the whole space $\R^d$ if one integrates
with respect to $x$ and over the whole half-line $\R^+=[0,\infty)$ if
  the integration is with respect to $t$.  As usual,
  $w_+=\max\{0,w\}$, $w_-=\max\{0,-w\}$, so $w=w_+ - w_-$.  Constants
  (always independent of $x$ and $t $) will be denoted by the same
  letter $C$, even if they may vary from line to line.  Sometimes we
  write, {\it e.g.}, $C=C(p,q,r)$ when we want to emphasize the
  dependence of $C$ on particular parameters~$p,\,q,\,r$, for
  instance.

\section{Main results}

In this work, we show two main results: we construct explicit
self-similar solutions of equation \eqref{eq:main}, as well as we
prove that the initial value problem \eqref{eq:main}--\eqref{eq:ic}
has a global-in-time weak solution which satisfies certain optimal decay
estimates.

We first ``recall'' the appropriate notion of weak solutions for
Equation~\eqref{eq:main}, see for instance \cite{V2,V1}.
\begin{definition}[Weak solutions] 
\label{defi:sol} 
A function $u: Q_T \to \R$ is a~\emph{weak solution} of
the problem \eqref{eq:main}--\eqref{eq:ic} in $Q_T$ if $u \in L^1
(Q_T)$, $\nabla^{\alpha-1} (|u|^{m-2}u) \in L_{\rm {loc}}^1(Q_T)$ and 
$|u|\nabla^{\alpha-1} (|u|^{m-2}u) \in L_{\rm {loc}}^1(Q_T)$, 
and
\[\iint  \left(u\partial_t\vphi - |u|\nabla^{\alpha-1}(|u|^{m-2}u)
\cdot \nabla \vphi\right) \dt \dx + \int u_0 (x) \vphi (0,x) \, {\rm
  d}x =0\] for all test functions $\vphi\in
\mathcal{C}^\infty(Q_T)\cap \mathcal{C}(\overline{Q_T})$ such that
$\vphi$ has a compact support in the space variable $x$ and vanishes
near $t=T$.
\end{definition}
The first main result of this work says that there is a family of
nonnegative {\em explicit} compactly supported self-similar solutions
of \eqref{eq:main}, {\it i.e.} nonnegative solutions that are
invariant under a suitable scaling.  Observe that if $u(t, x)$ is a
solution of \eqref{eq:main}, then so is $L^{d\lambda}u(L t,L^\lambda
x)$ for each $L>0$, where $\lambda=({d(m-1)+\alpha})^{-1}$.  Thus, the
scale invariant solutions should be of the following form
\begin{equation}\label{eq:self-sim-shape}
u(t,x) = \frac{1}{t^{d\lambda}}\Phi  \left( \frac{x}{t^{\lambda}} \right)
\qquad \text{with} \quad \lambda = \frac1{d(m-1) + \alpha},
\end{equation}
for some function $\Phi:\R^d\to \R$ satisfying the following nonlocal
``elliptic type'' equation
\begin{equation}\label{eq:ellip-sim}
  -\lambda \nabla \cdot (y \Phi) = \nabla \cdot (|\Phi| \nabla^{\alpha-1}( 
  |\Phi|^{m-2}\Phi ) )
 \qquad \text{where}  \quad y=\frac{x}{t^{\lambda}}.
\end{equation}
\begin{theorem}[Self-similar solutions]\label{thm:self-sim} 
Let $\alpha\in(0,2]$, $m>1$. 
Consider the function $\Phi_{\alpha,m} : \R^d \to \R$ defined as 
\begin{equation}\label{Phi}
\Phi_{\alpha,m} (y) = \left(k_{\alpha,d} (1-|y|^2)_+^{\frac{\alpha}{2}}\right)^\frac{1}{m-1}
 \end{equation}
 with the constant 
\[k_{\alpha,d} = \frac{d \,\Gamma\left(\frac{d}2\right)}{(d(m-1)+\alpha)2^\alpha
   \Gamma \left(1+\frac{\alpha}2\right)\Gamma
   \left(\frac{d+\alpha}2\right)}. \] Then, the function $u :
(0,\infty) \times \R^d \to \R^+$ defined by \eqref{eq:self-sim-shape}
with $\Phi=\Phi_{\alpha,m}$ is a weak solution of \eqref{eq:main} in
the sense of Definition~\ref{defi:sol} in $Q_{\eta,T}\equiv
(\eta,T)\times\R^d$ for every $0<\eta<T<\infty$. Moreover, $u(t,x)$
satisfies the equation in the pointwise sense for $|x|\neq
t^{\frac{1}{d(m-1)+\alpha}}$, and is
$\min\big\{\frac{\alpha}{2(m-1)},1\big\}$-H\"older continuous at the
interface $|x|=t^{\frac{1}{d(m-1)+\alpha}}$.
\end{theorem}
\begin{remark}
  When $\alpha=2$ in expression \eqref{self} below, we recover the classical
  Barenblatt--Kom\-pa\-neets--Pattle--Zel'dovich solutions of the porous medium
  equation \eqref{PME}, see for instance \cite{V1,V2}.
\end{remark}
\begin{remark} 
 For each $M\in(0,\infty)$ we can find a nonnegative self-similar
 solution $u$ with prescribed mass $M\equiv \int u(t,x)\, {\rm d} x$
 (which is conserved in time) by a suitable scaling of the profile
 $\Phi_{\alpha,m}$. Indeed, this self-similar solution is given by the
 formula
\begin{equation} 
  u(t,x)=t^{-\frac{d}{d(m-1)+\alpha}} \left(k_{\alpha,d} 
    \left(R^2-\left|x t^{-\frac{1}{d(m-1)+\alpha}}\right|^2\right)_ +^{\frac\alpha2}\right)^{\frac{1}{m-1}},
  \label{self}
\end{equation}  
where, for each $M>0$, there exists a {unique} $R>0$ such that $\int u(t,x)\dx=M$. 
\end{remark} 
\begin{remark}
Self-similar solutions of equation \eqref{eq:cv} (which is a
particular case of equation \eqref{eq:main}) have been proved to exist
in \cite{cv10b} by studying the following obstacle problem for the
fractional Laplacian. For $\alpha\in(0,2)$ and
$\Psi(y)=C-a|y|^2$ where $a=a(d,\alpha)$ and $C>0$, one looks for a function
$P=P(y)$ with the following properties:
\[ P\ge \Psi,\qquad (-\Delta)^{\frac{\alpha}{2}}P\ge 0,\qquad 
 \text{and } \quad \text{ either} \quad P=\Psi\quad \text{or} \quad
 (-\Delta)^{\frac{\alpha}{2}}P=0.\] The novelty of our approach is
 that we exhibit the {\it explicit} self-similar profile
 $\Phi_{\alpha,2}$ defined in \eqref{Phi} and, consequently, the {\it
   explicit} solution of this obstacle problem: $P(y)={\mathcal
   I}_\alpha\left(\Phi_{\alpha,2}\right)\left(\frac{y}{R}\right)$,
 where ${\mathcal I}_\alpha=(-\Delta)^{-\frac{\alpha}{2}}$ is the
 Riesz potential and $R>0$ is a suitable constant.
\end{remark}

Next, we prove the existence of weak solutions to the initial value problem
\eqref{eq:main}--\eqref{eq:ic}.
\begin{theorem}[Existence and decay of $L^p$-norms] \label{thm:hyper} 
Let $\alpha\in(0,2)$ and 
\begin{equation}\label{cond:m} 
\begin{cases}  m>1 +\frac{1-\alpha}d & \text{ if } \alpha \in (0,1], \\ 
m> 3 - \frac{2}\alpha & \text{ if } \alpha \in
(1,2). \end{cases} 
\end{equation}
Given $u_0\in L^1(\R^d)$, there exists a
global-in-time weak solution $u$ of the Cauchy problem
\eqref{eq:main}--\eqref{eq:ic}.  Moreover,
\[\int u(t,x) \dx=\int u_0(x) \dx\]
and 
\begin{equation}\label{u:decay}
\|u(t)\|_p\le C(d,\alpha,m) \|u_0\|_1^{\frac{d(m-1)/{p}
    +{\alpha}}{d(m-1)+\alpha}}
t^{-\frac{d}{d(m-1)+\alpha}\big(1-\frac{1}{p}\big)} \qquad \text{for
  all}\quad t>0,
\end{equation}
holds with the constant $C(d,\alpha,m)$ independent of $p$ and $u_0$.

The solution $u$ is nonnegative if the initial condition $u_0$ is so. 
If $u_0 \in L^p(\R^d)$ for some  $p \in [1,\infty]$, then
\[\|u(t)\|_p\leq \|u_0\|_p\] 
holds for all $t>0$. 
\end{theorem}
\begin{remark}
  Estimates~\eqref{u:decay} are sharp since the decay in
  Theorem~\ref{thm:hyper} corresponds exactly to that for self-similar
  solutions constructed in Theorem~\ref{thm:self-sim}.  Moreover, for
  $\alpha=2$, they are similar to those for degenerate partial
  differential equations like the porous medium equation (showing the
  regularization effect on the $L^p$-norms of solutions); see, {\it
    e.g.}, \cite{V1}, \cite[Ch. 2]{CJMTU}.
\end{remark}
\begin{remark}
After proving those hypercontractivity estimates in \cite{BIK}, we
learned that a~similar result is obtained in \cite{CSV}, however, for a less general model: $\alpha\in (0,1)$, $m=2$, and nonnegative $u_0$. 
 Moreover,
analogous decay estimates for another fractional porous medium
equation of the form
$\partial_tu+(-\Delta)^{\frac{\alpha}{2}}(|u|^{m-1}u)=0$ were proved
recently in \cite{dPQRV-CPAM}.
\end{remark}
Compared with the methods used in \cite{cv10a}, we propose an
alternative strategy of the proof of the existence of solutions.  In
this paper, we consider approximating solutions $u=u^{\delta,\e}$ of
the equation
\[\partial_tu=\delta\Delta u+\nabla\cdot(|u|\nabla^{\alpha-1}G_\e(u)),\]
considered in the whole space $\R^d$, where $G_\e(u)$ is a sufficiently smooth
approximation of $u|u|^{m-2}$, and then we pass to the limit with the
parameters $\e\searrow 0$, and $\delta\searrow 0$.
Solutions of the approximating equation exist because the parabolic
regularization term $\delta\Delta u$ is strong enough to regularize
equation \eqref{eq:main} when $0<\alpha< 2$, but of course not for $\alpha=2$.
Our approach resembles the approach to the one-dimensional model
achieved in \cite[Sec. 4 and 5]{bkm} via viscosity solutions.

\section{Preliminaries}
\label{sec:prelim}

In this section, we collect known results that we will used in proofs
of the main theorems.

\subsection*{Bessel and  hypergeometric functions}

{\it Bessel functions} of order $\nu$ are denoted by $J_\nu(z)$, and they
behave for small and large values of the (complex) variable~$z$ like
\begin{eqnarray*}
J_\nu (z) \sim & \frac1{\Gamma(\nu+1)} (\frac{z}2)^\nu & \text{
  as\ \ } z \to 0, \\ J_\nu (z) \sim &
(\frac2\pi)^{\frac12}\cos(z-\frac{\nu \pi}2 -
\frac{\pi}4){z^{-\frac12}} & \text{ as\ \ } |z| \to \infty,
\end{eqnarray*}
where for functions $f$, $g$,\ \ the relation $f \sim g$ means that $\frac{f}{g}
\to 1$.  For the proofs of those properties of $J_\nu$, the reader is
referred to, {\em e.g.}, \cite{wat}.

{\it The hypergeometric function}, denoted by ${}_2 F_1 (a,b;c;z)$, is
defined for complex numbers $a,b,c$ and $z$ as the sum of the series
\[{}_2 F_1 (a,b;c;z)=\sum_{n=0}^\infty\frac{(a)_n(b)_n}{(c)_nn!}z^n\ \ \ {\rm
  for\ \ }|z|<1,\] where $(a)_n\equiv \frac{\Gamma(a+n)}{\Gamma(a)}$,
and $\Gamma$ denotes the Euler Gamma function.  This series is
absolutely convergent in the open unit disc and also on the circle
$|z|=1$ if $\Re (a+b-c) <0$.

It is known \cite[p. 39]{mos} that when $b=-n$ is a negative integer, ${}_2
F_1(a,-n;c;z)$ is a~polynomial function of degree $n$.  In particular, we
have
\begin{equation}\label{eq:hyper-poly}
{}_2F_1(a,-1;c;z)= 1 - \frac{a}c z.  
\end{equation}
We will also use the following differentiation formula \cite[p. 41]{mos}
\begin{equation}\label{eq:hyper-diff}
\frac{d}{dz}\big({}_2 F_1 (a,b;c;z)\big) = \frac{ab}{c} \; {}_2 F_1 (a+1,b+1;c+1;z).
\end{equation}

{ \it The Weber--Schafheitlin integral.}  
If $0 < b < a$ and if integral
\eqref{eq:ws} below is convergent, then the following identity holds true
\begin{equation}\label{eq:ws}
  \int_0^{\infty} t^{-\lambda} J_\mu (at) J_\nu (bt) \, {\rm d}t 
  = \frac{b^{\nu} 2^{-\lambda} a^{\lambda-\nu-1} \Gamma
    (\frac{\nu+\mu-\lambda+1}{2})}{\Gamma (\frac{-\nu+\mu+\lambda+1}{2})
    \Gamma (1 + \nu)} {}_2 F_1
 \left(\frac{\nu+\mu-\lambda+1}{2},\frac{\nu-\mu-\lambda+1}{2}; \nu+1;\frac{b^2}{a^2}\right).
\end{equation}
According to Watson, \cite[pp. 401--403]{wat}, this result was obtained by
Sonine and Schafheitlin. However, it is usually referred to as the {\em
  Weber--Schafheitlin discontinuous integral}  
  since there occurs a discontinuity for $a=b$.

\subsection*{The Stroock--Varopoulos inequality}

We next recall the {\em the Stroock--Varopoulos inequality}, see
\cite[Theorem~2.1 and Condition~(1.7)]{LS}
for a proof.  
\begin{proposition}\label{prop:SV}
For  $\alpha\in(0,2]$, $w\in \mathcal{C}^\infty_c(\R^d)$ and $q>1$,
the following inequality holds true 
\begin{equation}\label{SV} 
\int \sgn w \,|w|^{q-1} \, (-\Delta)^{\frac\alpha2}w  \dx \ge
\frac{4(q-1)}{q^2}\int\left|\nabla^{\frac\alpha2}\left(\sgn w\,
|w|^{\frac{q}2}\right)\right|^2 \, \dx \ge
\frac{4(q-1)}{q^2}\int\left|\nabla^{\frac\alpha2} |w|^{\frac{q}2}\right|^2 \, \dx.
\end{equation} 
\end{proposition}

\subsection*{Fractional order Sobolev spaces}
The fractional order Sobolev spaces are defined as
\[ H^{s,p}(\R^d)=\{v\in L^p(\R^d): \nabla^sv\in L^p(\R^d)\}=\{v\in
L^p(\R^d) : (I-\Delta)^{\frac{s}{2}}v \in L^p(\R^d)\},\] here with
$p\in(1,\infty)$, supplemented with the usual norm denoted by
$\|\cdot\|_{H^{s,p}}$, and we refer the reader to the books
\cite{Taylor-tools, TaylorIII} for properties of those spaces. In
particular for $s=\alpha-1$ with $\alpha\in(1,2)$, the following
well-known continuous embedding will be used repeatedly
\begin{equation}\label{Hbound} 
H^{\alpha-1,p}(\R^d)\subset L^\infty(\R^d)
\qquad\text{provided}\  \quad p>\frac{d}{\alpha-1} \quad (>1). 
\end{equation}  
We also recall the fractional integration theorem  \cite[Ch.~V,
  \S 1.2]{Stein}: the Riesz potential ${\mathcal
  I}_{s}=(-\Delta)^{-\frac{s}{2}}$ satisfies
\begin{equation}\label{Riesz} \|{\mathcal I}_s u\|_q\leq C(p,q,s)\|u\|_p\end{equation}
for all $s\in (0,d)$ and $p,\,q\in (1,\infty)$ satisfying 
 $\frac1{q}=\frac1{p}-\frac{s}{d}$.

\subsection*{Some functional inequalities}

We will  use the following \emph{Nash inequality}
\begin{equation} 
\|v\|_2^{2\big(1+\frac\alpha{d}\big)}\le
C_N\|\nabla^{\frac\alpha2}v\|^2_2\|v\|_1^{\frac{2\alpha}d}\label{N}
\end{equation} 
valid for all functions $v\in L^1(\R^d)$, such that
$\nabla^{\frac\alpha2}v\in L^2(\R^d)$, and with a constant
$C_N=C(d,\alpha)>0$.  The proof of \eqref{N} for $d=1$ can be found
in, {\it e.g.}, \cite[Lemma 2.2]{KMX}, and this extends easily to the
general case $d\ge 1$.

Moreover, we will use the following \emph{Gagliardo--Nirenberg type
  inequality}
\begin{lemma}\label{GN} 
For $p>1$ and $p\ge m-1$, the inequality 
\begin{equation}
\|u\|_p^a\le C_N\left\|\nabla^{\frac\alpha2}|u|^{\frac{r}2}\right\|^2_2\|u\|_1^b 
\label{inter}
\end{equation} 
holds with 
\begin{equation}
a=\frac{p}{p-1}\frac{d(r-1)+\alpha}{d}, \qquad
b=a-r=\frac{d(m-1)+p\alpha}{d(p-1)}, \qquad r=p+m-1.\label{a-b}
\end{equation} 
\end{lemma} 
\begin{proof}
This inequality is a consequence of the Nash inequality \eqref{N} written
for $v=|u|^{\frac{r}2}$, {\it i.e. }
\begin{equation}\label{Nash-r}
  \|u\|_r^{r\big(1+\frac\alpha{d}\big)}\le
  C_N\left\|\nabla^{\frac\alpha2}|u|^{\frac{r}2}\right\|^2_2\|u\|_{\frac{r}2}^{\frac{r\alpha}{d}},
\end{equation}
and two H\"older inequalities
\[\|u\|_p\le \|u\|_r^\gamma\|u\|_1^{1-\gamma} \qquad \text{with} \quad
\gamma=\left(\frac{p-1}{r-1}\right)\frac{r}{p},\]
and
\[\|u\|_{\frac{r}2}\le \|u\|_p^\delta\|u\|_1^{1-\delta}\qquad \text{with}
\quad \delta=\left(\frac{r-2}{p-1}\right)\frac{p}{r}.\] 
Combining the above three inequalities, we get \eqref{inter}.  
\end{proof}

\section{Proof of Theorem~\ref{thm:self-sim}}
\label{sec:self-sim}

This section is devoted to the study of nonnegative self-similar
solutions for \eqref{eq:main} with $m>1$.  As explained above, this
problem reduces to a study of the elliptic-like
equation~\eqref{eq:ellip-sim} which for 
nonnegative  $\Phi$ takes the form 
\[-\lambda y\, \Phi = \Phi\nabla^{\alpha-1}( \Phi^{m-1}). \]
 Moreover, since we want to construct compactly
supported solutions, we are interested in solutions $\Phi$ vanishing
outside the unit ball $B_1$. This is the reason why we consider the
Dirichlet problem
\begin{equation}\label{Dirichlet}
\begin{split}
-\lambda y  = \nabla^{\alpha-1} (\Phi^{m-1}) & \quad \text{ in\ \ \ } B_1 ,\\
 \Phi=0  &\quad \text{ in\ \ \ } \R^d\setminus B_1 . 
 \end{split}
\end{equation}
It is well known that, in the case of nonlocal operators (such as
$\nabla^{\alpha-1}$), the homogeneous Dirichlet condition should be
understood in the form $ \Phi \equiv 0$ outside the domain $B_1$, and
not only $\Phi=0$ on the boundary $\partial B_1$. The reader is
referred to, {\it e.g.}, \cite{bci} for more {explanations}.

We claim that the proof of Theorem~\ref{thm:self-sim} reduces to the
following key computation. Here, ${}_2F_1$ denotes the classical
hypergeometric function defined in Section~\ref{sec:prelim}. 
\begin{lemma}\label{lem:getoor}
For all $\beta\in(0,2)$, $\beta <d$, and $\gamma>0$, we have 
\begin{equation}\label{eq:getoor}
{\mathcal I}_\beta\left((1-|y|^2)_+^{\frac{\gamma}{2}}\right) 
=\left\{
  \begin{array}{ll} 
    C_{\gamma,\beta,d} \times
    {}_2F_1\left(\frac{d-\beta}{2},-\frac{\gamma+\beta}{2};\frac{d}{2};|y|^2\right) 
    & {\rm for\ \  } |y| \le 1, \\
    \tilde{C}_{\gamma,\beta,d}\, |y|^{\beta-d} 
    \times {}_2F_1\left(\frac{d-\beta}{2},\frac{2-\beta}{2};\frac{d+\gamma}{2};\frac{1}{|y|^2} \right) 
    & {\rm for\ \  } |y| > 1,
  \end{array}
\right.
\end{equation}
with $C_{\gamma,\beta,d} = 2^{-\beta} \frac{\Gamma
  \left(\frac\gamma2+1\right)\Gamma \left(\frac{d-\beta}2\right)}{\Gamma
  \left(\frac{d}2\right)\Gamma \left(\frac{\beta+\gamma}2+1\right)}$ and
$\tilde{C}_{\gamma,\beta,d} = 2^{-\beta} \frac{\Gamma
  \left(\frac\gamma2+1\right)\Gamma \left(\frac{d-\beta}2\right)}{\Gamma
  \left(\frac{d}4\right)\Gamma \left(\frac{d+\gamma}2+1\right)}$.
\end{lemma}

The proof of this lemma is postponed to the end of this section. 

The following corollary is an immediate consequence of Lemma
\ref{lem:getoor} with $2-\beta=\alpha=\gamma$, of the property of
${}_2F_1$ formulated in \eqref{eq:hyper-poly}, and of the identity
$(-\Delta)^{\frac{\alpha}{2}}=(-\Delta){\mathcal I}_{2-\alpha}$.  It
has an important probabilistic interpretation, and recently, related
results and generalizations have been proved in \cite{Dyda}.
\begin{corollary}[Getoor {\cite[Th. 5.2]{getoor}}] \label{lem:getoor-genuine} For
all $\alpha \in (0,2]$, the identity
\[K_{\alpha,d}(-\Delta)^{\frac\alpha2} (1-|y|^2)_+^{\frac{\alpha}{2}} =
1\ \ {\rm in}\ \ B_1\]
holds true with the constant
$K_{\alpha,d}=\frac{\Gamma\left(\frac{d}2\right) }{2^\alpha
  \Gamma\left(1+\frac\alpha2\right) \Gamma \left(\frac{d+\alpha}2\right)}$.
\end{corollary}
\bigskip

Before proving Lemma~\ref{lem:getoor},  we first use it to derive  Theorem~\ref{thm:self-sim}.
\begin{proof}[Proof of Theorem~\ref{thm:self-sim}]
  We check that $u(t,x) = t^{-d\lambda} \Phi_{\alpha,m}
  (t^{-\lambda}x)$ is a weak solution of \eqref{eq:main} in the sense
  of Definition~\ref{defi:sol}.  First, $u \in L^1 (Q_T)$ if and only
  if $\Phi_{\alpha,m} \in L^1 (\R^d)$, which is obviously true. For later use,
  it is convenient to introduce the function 
$\Phi_\alpha(u)= \left(1-|y|^2\right)_+^\frac{\alpha}{2}$, 
so that $k_{\alpha,d} \Phi_\alpha=\Phi_{\alpha,m}^{m-1} $.

The fact that, for all $\eta,T$ such that $0 < \eta < T$,
$u\nabla^{\alpha-1}(u^{m-1})$ and $\nabla^{\alpha-1}(u^{m-1})$ are
locally integrable in $(\eta,T) \times \R^d$ follows from
\[{\mathcal I}_{2-\alpha}(\Phi_{\alpha}) \in H^{1,1}_{\rm {loc}} (\R^d),\]
which we prove by computing ${\mathcal I}_{2-\alpha} (\Phi_\alpha)$. In order to
do so, we first assume that $\alpha > 2 -d$, and we apply
Lemma~\ref{lem:getoor} with $\gamma = \alpha \in (0,2)$ and $\beta = 2
-\alpha$, we use equation \eqref{eq:hyper-poly}, and we get
\begin{equation}\label{eq:ialpha}
\qquad  {\mathcal I}_{2-\alpha} (\Phi_\alpha) (y)
= \left\{\begin{array}{ll}
      C_{\alpha,2-\alpha,d} \left( 1 - \frac{d+\alpha-2}{d} |y|^2 \right) & \text{ if\ \ } |y| \le 1,\\
      \tilde{C}_{\alpha,2-\alpha,d} |y|^{2-(d+\alpha)}\, {}_2F_1
      \left(\frac{d+\alpha}2 -1,\frac\alpha2;\frac{d+\alpha}2;
        \frac{1}{|y|^2}\right) & \text{ if\ \  } |y| > 1.
\end{array}\right. 
\end{equation}
The right-hand side of equation \eqref{eq:ialpha} defines a locally
integrable function because
\[ a+b-c = (\frac{d+\alpha}{2}-1)+\frac{\alpha}{2}-\frac{d+\alpha}{2}<0.\]  
We then
deduce that
\[\nabla^{\alpha -1} (\Phi_\alpha) (y) = \nabla {\mathcal I}_{2-\alpha}
(\Phi_\alpha)(y)= - \frac{\lambda}{k_{\alpha,d}} y \ \ {\rm
  for}\ \ y\in B_1.\] 
Note  also that $\nabla^{\alpha -1} (\Phi_\alpha)$ can be computed
outside $B_1$ thanks to the differentiation
formula~\eqref{eq:hyper-diff}. 

We now remark that
\[\Phi_{\alpha,m} (y)\left(\nabla^{\alpha-1} \Phi_{\alpha,m}^{m-1}\right) (y)= -
\lambda y \,\Phi_{\alpha,m} (y) \ \ {\rm for\ all}\
\ y \in \R^d,\]
which is in $L^1 (\R^d)$. Moreover, the following equalities hold true in
the sense of distributions in $Q_T$,
\begin{align*}
\partial_tu (t,x) & = - \lambda t^{-d\lambda -1} \nabla_y \cdot (y
\Phi_{\alpha,m}) (t^{-\lambda} x),\\ \nabla_x\cdot( u
\nabla^{\alpha-1} (|u|^{m-1})) (t,x) &= t^{-d\lambda -1}
\nabla_y\cdot( \Phi_{\alpha,m} \nabla^{\alpha-1}\Phi_{\alpha,m}^{m-1})
(t^{-\lambda}x).
\end{align*}
This allows us to conclude that $u$ is indeed a weak solution of
\eqref{eq:main} in $(\eta,T)\times\R^d$ for all $0<\eta<T<\infty$ if
$\alpha > 2 -d$. 

Assume now that $0 <\alpha \le 2 -d$, which  means that
  $d=1$ and $\alpha \le 1$. The critical case $\alpha =1$ can be
obtained by passing to the limit as $\alpha \searrow 1$; indeed, the
constants $C_{\alpha,2-\alpha,d}(d+\alpha-2)$ and
$\tilde{C}_{\alpha,2-\alpha,d}(d+\alpha-2)$ appearing in
\eqref{eq:ialpha} simplify thanks to the relation $z \Gamma (z) =
\Gamma(z+1)$. If now $\alpha <1$, we can argue as above by analytic
continuation. The proof is now complete.
\end{proof}

\bigskip

Now we turn to the proof of the main technical lemma. 
\begin{proof}[Proof of Lemma~\ref{lem:getoor}]
  We first assume that $\beta \in (0,d)$ and that $\gamma > \max\{0,d-
  2\beta -1\}$, and we then argue by the analytic continuation with
  a choice of parameters corresponding each time to ${}_2F_{1}$
  defined and bounded for all $|y|\le 1$.
    
  The Fourier transform of $\Phi_\gamma (y)
  =\left(1-|y|^2\right)^{\frac\gamma2}_+$ is expressed in terms of Bessel
  functions, see, {\em e.g.}, \cite[Ch. IV, Sec. 3]{Stein}
\[
\mathcal{F}( \Phi_\gamma) (\xi) = 2^{\frac\gamma2}\Gamma
\left(\frac{\gamma}2+1\right) \frac1{|\xi|^{\frac{d+\gamma}2}}
J_{\frac{d+\gamma}2} (|\xi|).
\]
Since ${\mathcal I}_\beta$ is the Fourier multiplier of symbol $|\xi|^{-\beta}$, ${\mathcal I}_\beta
(\Phi_\gamma)$ is the (inverse) Fourier transform of the following
radially symmetric function
\[
  2^{\frac\gamma2}\Gamma
\left(\frac{\gamma}2+1\right) \frac1{|\xi|^{\frac{d+\gamma}2+\beta}}
J_{\frac{d+\gamma}2} (|\xi|).
\]
We recall that by properties of Bessel functions collected in
Section~\ref{sec:prelim}, we have $J_\nu(r) = {\mathcal O}
\left(r^{\nu}\right)$ as $r \to 0$ and $J_\nu(r) = {\mathcal
  O}\left(r^{-\frac12}\right)$ as $r \to \infty$. We see that the
previous function is integrable since $\beta <d$ and $d < \gamma+
2\beta +1$.  Thanks to \cite[Th.~3.3]{stein-weiss}, we get
\begin{eqnarray*}
{\mathcal I}_\beta (\Phi_\gamma) (y) &=& 2^{\frac\gamma2}\Gamma \left(\frac{\gamma}2+1\right)
|y|^{1 - \frac{d}2} \int_0^{\infty} \frac{1}{t^{\frac{d+\gamma}2+\beta}}
J_{\frac{d+\gamma}2}(t) \,t^{\frac{d}2} J_{\frac{d}2-1} (t|y|) \,{\rm d}t \\
& = &  2^{\frac\gamma2}\Gamma \left(\frac{\gamma}2+1\right)
|y|^{1 - \frac{d}2} \int_0^{\infty} t^{-(\frac{\gamma}2+\beta)}
J_{\frac{d+\gamma}2}(t) J_{\frac{d}2-1} (t|y|) \,{\rm d}t.
\end{eqnarray*}

\noindent
We obtain \eqref{eq:getoor} applying \eqref{eq:ws} with the following
choice of parameters: 
\begin{itemize}
\item if $|y| \le 1$, we put \ \ $\lambda = \frac\gamma2+\beta$, $\mu =
\frac{d+\gamma}2$, $\nu=\frac{d}2-1$, $a=1$ and $b=|y|$, 

\item if $|y| > 1$, we put \ \ $\lambda = \frac\gamma2+\beta$, $\mu =
\frac{d}2-1$, $\nu=\frac{d+\gamma}2$, $a=|y|$ and $b=1$.  
\end{itemize} 
\end{proof}

\section{A regularized problem}

In order to construct weak solutions of \eqref{eq:main} for general
initial data, we first consider the following regularized problem
\begin{equation}\label{reg}
\partial_t u =\delta\Delta u+\nabla\cdot\left(|u|\nabla^{\alpha-1}(G(u))\right),\ \ \ u(0,x)=u_0(x), 
\end{equation} 
where $G: \R \to \R$ satisfies
\begin{equation}\label{hyp:G}
\left\{\begin{array}{l}
 G \text{ differentiable and increasing}, \\ G(0)=G'(0)=0, \\  
G' \text{ locally Lipschitz continuous}. 
\end{array}\right.
\end{equation}
 Remark that for $m \ge 3$ or $m=2$, the function $G(u)=|u|^{m-2}u$ satisfies
 \eqref{hyp:G}. For $m \in (1,3)$, we consider the following  approximation $G=G_\e$ of $|u|^{m-2}u$
\[
G_\eps(u)=\sgn u\left(\left(u^2+\eps^2\right)^{\frac{m-1}2}-\eps^{m-1}\right)
\]
with $\eps >0$. 
The following theorem holds true for a general function $G$ satisfying \eqref{hyp:G}.
\begin{theorem}[Existence of solutions to the regularized problem] \label{thm:reg} 
Let $\delta >0$ and assume that $G$ is an arbitrary function  satisfying \eqref{hyp:G}. Moreover,
assume
\begin{equation}\label{icc}
u_0\in \begin{cases}
L^1(\R^d)\cap L^\infty(\R^d) & \text{if}  \quad \alpha\in(0,1],\\
L^1(\R^d)\cap \big(\cap_{p > p_\alpha} H^{\alpha-1,p}(\R^d)\big) 
& \text{if} \quad \alpha\in(1,2),
\end{cases}
\end{equation}
with $p_\alpha =\frac{d}{\alpha-1}>1$. 
There exists a unique function $u$ in the space
\begin{equation}\label{rsp}
u\in 
\left\{
\begin{array}{lcc}
{\mathcal C}\left([0,\infty),L^1(\R^d)\cap L^\infty(\R^d)\right)& \text{if}\  \; \alpha\in(0,1],\\
\cap_{p > p_\alpha} {\mathcal C}\left([0,\infty),L^1(\R^d)\cap 
  H^{\alpha-1,p}(\R^d)\right) & \text{if}\ \; \alpha\in(1,2),
\end{array}\right.
\end{equation} 
satisfying problem~\eqref{reg} in the usual weak sense
\begin{equation}\label{distrib:reg}
\iint \left(u\partial_t \varphi -
|u|\nabla^{\alpha-1} (G(u))\cdot\nabla\varphi -\delta \nabla u\cdot\nabla \varphi \right) \dt \dx =0
\end{equation}
for all  $\varphi\in \mathcal{C}^\infty_c (Q_T)$.

  Moreover, $u(t,x)$ is nonnegative if the initial
condition $u_0$ is so, and for all $t>0$ and  $q\in
[1,\infty]$ we have
\begin{equation}\label{thm:mass}
\int u(t,x) \dx=\int u_0(x) \dx\qquad \text{and} \qquad
\|u(t)\|_q\leq \|u_0\|_q.
\end{equation}
\end{theorem}

\subsection*{Local-in-time  existence of mild solutions}

\begin{proposition}
\label{prop:loc-time-ex}
Let  $p>p_\alpha =\frac{d}{\alpha-1}$.  
There exists $T>0$  depending only on $u_0$, and a function $u$ in the space
\begin{equation}\label{rsp-loc}
u\in 
\left\{
\begin{array}{lcc}
{\mathcal C}\left([0,T),L^1(\R^d)\cap L^\infty(\R^d)\right)& \text{if}\  \; \alpha\in(0,1],\\
{\mathcal C}\left([0,T),L^1(\R^d)\cap H^{\alpha-1,p}(\R^d)\right) & \text{if}\ \; \alpha\in(1,2)
\end{array}\right.
\end{equation} 
such that
\begin{equation}\label{Duh}
 u (t) ={\rm  e}^{\delta t\Delta} u_0 + \int_0^t \nabla {\rm e}^{\delta(t-s) \Delta}\cdot
\Psi (u(s)) \ds \qquad {\rm with}  \quad \Psi (u) = |u| \nabla^{\alpha-1} G(u),
\end{equation}
in $\mathcal{C}[0,T],L^1 (\R^d) \cap L^\infty (\R^d))$  where
 ${\rm e}^{t \Delta}$ denotes the heat semigroup.
\end{proposition}
\begin{remark}
We identify the heat semigroup ${\rm e}^{t\Delta}$ and its kernel
$(4\pi t)^{-\frac{d}{2}}\exp\left(-\frac{|x|^2}{4t}\right)$. We will
use the following classical fact
\begin{equation} \label{lin} \left\|\nabla^\beta
{\rm e}^{\delta t\Delta}v \right\|_p\le C(p,r,\beta,\delta)
  t^{-\frac{d}{2}\left(\frac{1}{r}-\frac{1}{p}\right)-\frac{\beta}{2}}
  \|v\|_r
\end{equation}
with $1\leq r\leq p\leq \infty$, and $\beta \in [1,2)$.
\end{remark}
Now we turn to the proof of Proposition~\ref{prop:loc-time-ex}.
\begin{proof}[Proof of Proposition~\ref{prop:loc-time-ex}]
We look for a solution $u \in {\mathcal C}([0,T],X)$ as a fixed point of the map
\[ {\mathcal T}: u\mapsto {\rm e}^{\delta
  t\Delta}u_0 + \int_0^t \nabla{\rm e}^{\delta(t-s)\Delta}\cdot \Psi (u(s)) \ds, \]
where $X$ is chosen as follows
\begin{equation}\label{Xspace}
X = \begin{cases} L^1 (\R^d) \cap L^\infty (\R^d) & \text{ if }
  \alpha \in (0,1], \\ L^1 (\R^d) \cap H^{\alpha-1,p} (\R^d) & \text{
  if } \alpha \in (1,2).\end{cases} 
\end{equation} 
The associated norms are
$\|u\|_1 + \|u\|_Y$ with $Y= L^\infty (\R^d)$ and $Y=
H^{\alpha-1,p}(\R^d)$, respectively.  We show that $\mathcal{T}$ has a
fixed point by the Banach contraction principle as soon as
$T=T(\|u_0\|_X)>0$ is sufficiently small.

In both cases, it is enough to prove the  following lemma.
\begin{lemma}\label{lem:invariance+contrac}
For all $T \in (0,1)$, the operator $\mathcal{T}$ maps
$\mathcal{C}([0,T],X)$ into itself.  Moreover, there exist $C>0$ and
$\gamma >0$ such that for all $u,v \in \overline{B}(0,R) \subset
\mathcal{C}([0,T],X)$,
\begin{equation}
\label{estim:lip}
 \|\mathcal{T} (u) - \mathcal{T}(v)\|_{\mathcal{C}([0,T],X)} \le
 C_1(R) T^\gamma \|u-v\|_{\mathcal{C}([0,T],X)},
\end{equation}
where $C_1(R)$ is a constant which also depends on $\alpha, d,
 \eps, m, \delta$ {\rm (}and on $p$ if $\alpha \in (1,2)${\rm )}.
\end{lemma}
Indeed, once this lemma is proved, we first derive
\begin{equation}
\label{estim:linfty} 
\|\mathcal{T} (u)\|_{\mathcal{C}([0,T],X)} \le \|u_0\|_X + R C_1 (R) T^\gamma
\end{equation}
by choosing $v=0$ in \eqref{estim:lip}
and using estimate \eqref{lin}.
Now it is enough to choose $R = 2
\|u_0\|_X$ and $T>0$ such that $ C_1(R) T^\gamma \le \frac12$
in order to ensure that $\mathcal{T}$ maps $\overline{B}(0,R)$ into
itself, and is a contraction. \medskip

\noindent{\bf The case $\alpha \in (0,1]$.} \ 
In order to get estimate \eqref{estim:lip}, we first write
\begin{equation}\label{mapT}
{\mathcal T}(u)(t)-{\mathcal T}(v)(t)=\int_0^t \nabla {\rm e}^{\delta (t-s)\Delta}
\cdot\left( \Psi(u)-\Psi(v)\right)(s)\ds, 
\end{equation}
and  the difference of $\Psi$'s is represented as
\begin{equation}\label{differ}
\Psi(u)-\Psi(v)=(|u|-|v|)\nabla^{\alpha-1}G(u)+|v|\nabla^{\alpha-1}(G(u)-G(v)).
\end{equation}
\begin{lemma}
For every $\alpha \in (0,1]$ and  $p \in (1,\infty)$ 
there exists a constant $C(p,\alpha)>0$ such that 
for all  $u \in  L^\infty (\R^d)\cap L^p(\R^d)$  the following inequality
 \begin{equation}\label{Riesz2}
\|\nabla^{\alpha-1} (G(u)-G(v))\|_q\leq
C(p,\alpha)
\left(
\sup_{|z|\leq\|u\|_\infty+\|v\|_\infty}
 G'(z) 
\right)
\|u-v\|_p
\end{equation} 
holds true with $\frac1{q}=\frac1{p}-\frac{1-\alpha}{d}$. 
\end{lemma}
\begin{proof}
  Since, for $\alpha \in (0,1]$, we have
  $\nabla^{\alpha-1}= \nabla^0{\mathcal I}_{1-\alpha}$, where the
  components of $\nabla^0$ are the Riesz transforms (which are bounded
  operators on $L^p(\R^d)$ for each $p\in(1,\infty)$), we 
  obtain \eqref{Riesz2} from estimate \eqref{Riesz} as follows
\[ \|\nabla^{\alpha-1} (G(u)-G(v))\|_q \le \|\mathcal{I}_{1-\alpha} (G(u)-G(v))\|_q \le C(p,\alpha)\| G(u)-G(v)\|_p \le C \|u-v\|_p.\]
\end{proof}

Now, we come back to the proof of \eqref{estim:lip} with
$X=L^1(\R^d)\cap L^\infty(\R^d)$, First, for all $u,v \in
\overline{B}(0,R)\subset X$ and some $q \in (1,\infty)$,
\begin{align}
\|\Psi(u)-\Psi(v)\|_1&\le
\nonumber C \|u-v\|_{q^\ast}\|\nabla^{\alpha-1}G(u)\|_q+
\nonumber C \|v\|_{q^\ast}\|\nabla^{\alpha-1}(G(u)-G(v))\|_q\\  
\label{estim:key1} &\le C\|u-v\|_{q^\ast} G' (\|u\|_\infty) \|u\|_p 
+C \|v\|_{q^\ast}G' (\|u\|_\infty+\|v\|_\infty) \|u-v\|_p\\
\nonumber  &\le C_0 (R) \|u-v\|_X
\end{align}
with $C_0(R)=C(\alpha,d,q,RG' (2R))$ and $\frac1q+\frac1{q^\ast}=1$
and $\frac1q=\frac1p-\frac{1-\alpha}{d}$. We used estimate
\eqref{Riesz2} twice to get the second line in \eqref{estim:key1}, and
the inequality $\|u\|_r \le \|u\|_X$ which is valid for all $r \in
[1,\infty]$ to obtain the last one.

The estimate of the second norm in $X$  is obtained similarly:
for all $u,v \in \overline{B}(0,R) \subset X$ and   some  $q\in(1,\infty)$,
\begin{equation}\label{estim:key2}
 \|\Psi(u)-\Psi(v)\|_q  \le  \|u-v\|_\infty \|\nabla^{\alpha-1} G
(u)\|_q +  \|v\|_\infty \|\nabla^{\alpha-1}(G (u)-G (v))\|_q 
  \le C_0(R)  \|u-v\|_X.
\end{equation}
Now, we apply inequalities \eqref{lin} with $\beta=1$, $\langle
p,r\rangle =\langle 1,1\rangle$ and $\langle
p,r\rangle=\langle\infty,q\rangle$, respectively, and the estimates \eqref{estim:key1}, \eqref{estim:key2} yield
\begin{align}
\label{estim:key1+}
\|\mathcal{T} (u) - \mathcal{T}(v)\|_{\mathcal{C}([0,T],L^1(\R^d))}
\le C C_0 (R) T^{\frac12}  \|u-v\|_{\mathcal{C}[0,T],X)}, \\
\label{estim:key2+}
\|\mathcal{T} (u) - \mathcal{T}(v)\|_{\mathcal{C}([0,T],L^\infty(\R^d))}
\le C C_0 (R) T^{\frac12-\frac{d}2 \frac1{q^\ast}}  \|u-v\|_{\mathcal{C}[0,T],X)}.
\end{align}
 Combining \eqref{mapT}, \eqref{differ}, \eqref{estim:key1+} and
 \eqref{estim:key2+}, we thus get \eqref{estim:lip} for
 $\alpha\in(0,1]$, with $\gamma =
   \frac12-\frac{d}{2}\frac{1}{q^\ast}$, now with a new constant $C_1
   (R) = C (\alpha,d,q^\ast,RG' (2R))$, where we have chosen
   $q^\ast>d$ to ensure $\frac{d}{2}\frac{1}{q^\ast}<\frac12$.

As far as the continuity of $\mathcal{T}(u)$ with respect to time is
concerned, it is enough to study 
\[ \mathcal{S}(u) = \int_0^t \nabla {\rm e}^{\delta(t-s)\Delta} \cdot \Psi (u(s)) \ds.\]
We fix $t \in [0,T]$ and write for $h$ small enough (and positive if
$t=0$, negative if $t=T$), 
\[\mathcal{S} (u(t+h))- \mathcal{S}(u(t)) = \int_t^{t+h} \nabla
{\rm e}^{\delta(t+h-s)\Delta} \cdot \Psi (u(s)) \ds  + 
\int_0^t \nabla{\rm  e}^{\delta(t-s)\Delta} \cdot ({\rm e}^{\delta h \Delta}
   \Psi (u(s)) - \Psi (u(s))) \ds.
\]
As above, use   two key estimates 
\eqref{estim:key1} and \eqref{estim:key1}
together with \eqref{lin}
(and the dominated convergence theorem) to conclude the proof of Proposition \ref{prop:loc-time-ex} for $\alpha\in(0,1]$. 
\medskip

\noindent{\bf The case $\alpha \in (1,2)$.} \ We argue as before,
using  \eqref{mapT} and \eqref{differ}. We need now to
state and to prove the corresponding key technical lemma.
\begin{lemma}
For $\alpha \in (1,2)$, $p>\frac{d}{\alpha-1}$, and $u \in H^{\alpha-1,p} (\R^d)$,
\begin{equation}\label{Riesz3}
\|G(u)-G(v)\|_{H^{\alpha-1,p}} \le
C_2(\|u\|_{H^{\alpha-1,p}}+\|v\|_{H^{\alpha-1,p}}) \|u -v\|_{H^{\alpha-1,p}},
\end{equation}
where $C_2(\|u\|_{H^{\alpha-1,p}}+\|v\|_{H^{\alpha-1,p}})$ depends on
$\|u\|_{H^{\alpha-1,p}}$,  $\|v\|_{H^{\alpha-1,p}}$ and on the
$W^{1,\infty}$-norm of $G'$ in the interval $[0,\|u\|_\infty+\|v\|_\infty]$. 
\end{lemma}
\begin{proof}
We use the classical identity 
\begin{equation}\label{GGK}
  G(u)-G(v)=K(u,v)(u-v) \qquad 
\text{with}\quad K(u,v)=\int_0^1 G'(\tau u+(1-\tau)v) \,{\rm d}\tau,
\end{equation} 
 and the Moser estimate for the product of two functions in
 $H^{\alpha-1,p}(\R^d)$, see, {\it e.g.}, \cite[Ch. 2,
   ineq. (0.22)]{Taylor-tools}, to obtain the following inequality
\begin{equation}\label{non}
\|G(u)-G(v)\|_{H^{\alpha-1,p}} \le
\|u-v\|_\infty\|K(u,v)\|_{H^{\alpha-1,p}} +
\|u-v\|_{H^{\alpha-1,p}}\|K(u,v)\|_\infty.
\end{equation}
Moreover, we recall \cite[Ch. 2, Prop. 4.1]{Taylor-tools} that for
every increasing locally Lipschitz function $H$ we have
\[\|H(u)\|_{H^{\alpha-1,p}}\le C|H'|(\|u\|_\infty)\|u\|_{H^{\alpha-1,p}}\]
for every $p\in(1,\infty)$ and $\alpha-1\in(0,1)$. Choosing $H=G'$, we
deduce that
\begin{align}
\nonumber
\|K(u,v)\|_{H^{\alpha-1,p}} &\le \int_0^1 \|G'(\tau u + (1-\tau)v)
\|_{H^{\alpha-1,p}} \,{\rm d} \tau\\ 
\label{non2} & \le C |G''|
(\|u\|_\infty+\|v\|_\infty) ( \| u\|_{H^{\alpha-1,p}} + \|v
\|_{H^{\alpha-1,p}} ).
\end{align}
Moreover, we have the trivial estimate
\begin{equation}\label{non3}
 \|K(u,v)\|_\infty \le |G'| (\|u\|_\infty + \|v\|_\infty).
\end{equation}
Combining \eqref{non}, \eqref{non2}, \eqref{non3} and \eqref{Hbound}, we finally complete the proof of inequality \eqref{Riesz3}. 
\end{proof}
Now, we are in a position to obtain the estimate of the first component of the norm of $X$. From \eqref{differ}, we get
\begin{equation}
\label{key1} \|\Psi (u) - \Psi (v) \|_1   \le C_2(R)R \|u-v\|_{p'}   +
C_2 (2R) \|v\|_{p'}  \|u-v\|_X  \le C_3(R) \|u-v\|_X
\end{equation}
with $C_3(R) = CR( C_2(R) +  C_2 (2R))$.  We used inequality \eqref{Riesz3}
twice, as well as the fact that $\|u\|_r \le C \|u\|_X$ for all $r \in [1,\infty]$ and $u \in X$.
From \eqref{key1} and \eqref{lin} with $\beta =1$ and $\langle
p,r\rangle = \langle 1,1\rangle$, we  get inequality \eqref{estim:key1+} where
$C_0(R)$ is replaced with $\tilde{C}_0(R)=C_3(R)$. 

The  estimate of the $H^{\alpha-1,p}$-norm  is  obtained analogously. 
First, we have 
\begin{equation}
\label{key2} \|\Psi(u)-\Psi(v)\|_p   \le  \|u-v\|_\infty \|\nabla^{\alpha-1} G
(u)\|_p +  \|v\|_\infty \|\nabla^{\alpha-1}(G (u)-G (v))\|_p\\
   \le C_3(R)  \|u-v\|_X.
\end{equation}
From inequalities \eqref{key2} and \eqref{lin} with $\beta =1$ and $\beta=\alpha$ and $\langle
p,r\rangle = \langle p,p\rangle$, we  get
\begin{equation}\label{key2+}
\|\mathcal{T} (u) - \mathcal{T}(v)\|_{\mathcal{C}([0,T],H^{\alpha-1,p}(\R^d))}
\le C C_3 (R) (T^{\frac12}+ T^{\frac12-\frac\alpha2})  \|u-v\|_{\mathcal{C}[0,T],X)}.
\end{equation}
Finally, combining \eqref{estim:key1+} and \eqref{key2+}, we complete
the proof of \eqref{estim:lip} with $\gamma = \frac12-\frac\alpha2$
and with some $\tilde{C}_1 (R)$.

The time continuity of $\mathcal{T}(u)$ is proved as in the case
$\alpha\in(0,1]$, and this achieves the proof of
  Proposition~\ref{prop:loc-time-ex}.
\end{proof}

\subsection*{Regularity of the solutions}

\begin{corollary}[Regularity of the solutions] \label{cor:reg}
Consider $u_0 \in L^1 (\R^d) \cap L^\infty(\R^d)$ if $\alpha \in
(0,1]$ {\rm (}as before{\rm )}, and $u_0 \in L^1(\R^d) \cap
  \left(\cap_{p > p_\alpha} H^{\alpha-1,p} (\R^d) \right)$ if $\alpha
  \in (1,2)$ {\rm (}a strengthened assumption{\rm )}. Then the
  solution constructed in Proposition~\ref{prop:loc-time-ex} enjoys
  the following regularity
\begin{equation}\label{reg-sol}
u\in  {\mathcal
    C}^1\left((0,T), L^p(\R^d)\right) \cap {\mathcal
    C}\left((0,T), H^{1,p}(\R^d)\right)
\end{equation}
for every $p\in (\bar{p}_\alpha,\infty)$ with 
\[\bar{p}_\alpha=\begin{cases} 
\frac{d}{d-(1-\alpha)} & \text{ if } \alpha \in (0,1],\\
\frac{d}{\alpha-1} & \text{ if } \alpha \in (1,2).
\end{cases}\]  
In particular, $u=u(t,x)$ is a weak solution
  of equation \eqref{reg}, i.e.
\[ \partial_t u =\delta \Delta u+ \nabla \cdot \Psi(u) \qquad {\rm with}\ \ \Psi(u)=|u|\nabla^{\alpha-1}G(u) \]
in $(0,T) \times \R^d$ in the sense of distributions {\rm
  (}cf. equation \eqref{distrib:reg}{\rm )}, and
\begin{equation}\label{mass:cons}
\int u(t,x) \dx = \int u_0 (x) \dx 
\end{equation}
for all $t \in (0,T)$. 
\end{corollary}
\begin{proof}
If $\alpha\in (0,1]$, for every $u\in {\mathcal C}([0,T],X)$, 
where the space $X$ is defined in \eqref{Xspace}, we obviously have 
$u \in
  L^\infty ((0,T),L^\infty(\R^d))$. We derive from inequality 
\eqref{Riesz} that
  for all $p\in (\bar{p}_\alpha,\infty)$, we have $\nabla^{\alpha-1}
  G(u) \in L^\infty ((0,T),L^p(\R^d))$ with $\bar{p}_\alpha=
  \frac{d}{d-(1-\alpha)}>1$.  This implies
\[ \nabla \cdot \Psi(u) \in L^q\left((0,T),H^{-1,p}(\R^d)\right)\]
for all $q \in (1,\infty)$ and $p \in (\bar{p}_\alpha,\infty)$. 
 Thus, the maximal regularity
of mild solutions for the nonhomogeneous heat equation \cite{LSU}
gives us $\nabla u \in L^q((0,T),L^p(\R^d))$ for every $q \in
(1,\infty)$ and $p \in (\bar{p}_\alpha,\infty)$.  Consequently,
$\nabla \cdot \left(|u|\nabla^{\alpha-1}G(u)\right)=f_1+f_2$ with 
\begin{equation}\label{cont:grad}
f_1 =\nabla |u|\cdot
\nabla^{\alpha-1}G(u), \qquad f_2 = -|u| (-\Delta)^{\frac\alpha2} G(u). 
\end{equation}
First, we remark that $f_1 \in L^q((0,T),L^p(\R^d))$ for every $q \in
(1,\infty)$ and $p \in (\bar{p}_\alpha,\infty)$.  Second, we notice
that $G(u) \in L^q((0,T),H^{1,p}(\R^d))$, hence $
(-\Delta)^{\frac\alpha2} G(u) \in L^q((0,T),H^{1-\alpha,p}(\R^d))
\subset L^q((0,T),L^p(\R^d))$. Hence, we also have $f_2 \in
L^q((0,T),L^p(\R^d))$ for every $q \in (1,\infty)$ and $p \in
(\bar{p}_\alpha,\infty)$.  Using again the maximal regularity result,
we obtain that
\[\partial_t u \in L^q((0,T),L^p(\R^d))\]
for every $q \in (1,\infty)$ and $p \in (\bar{p}_\alpha,\infty)$.  Thus, 
using  the following representation in
$L^p(\R^d)$ for all $0<s<t$
\[u(t)-u(s)=\int_s^t \partial_tu(\tau)\, {\rm d}\tau,\] 
and the  H\"older inequality we obtain  for every $p \in(\bar{p}_\alpha,\infty)$ 
\[\|u(t)-u(s)\|_p\le \int_s^t\|\partial_tu(\tau)\|_p \dta \le
\left(\int_s^t\|\partial_tu(\tau)\|^q_p\dta\right)^{\frac1q}\ (t-s)^{\frac{1}{q^\ast}},\]
\textit{i.e.} $u \in \mathcal{C}^{0,\beta}((0,T),L^p (\R^d))$ for all
$p \in (\bar{p}_\alpha,\infty)$ and $\beta \in (0,1)$. Using inequality \eqref{Riesz2}, this estimate 
implies that $\Psi (u) \in \mathcal{C}^{0,\beta}((0,T),L^p (\R^d))$
for every $p \in (\bar{p}_\alpha,\infty)$.  Now, the classical theory of linear
parabolic equations, see, {\it e.g.}, \cite[Ch. 4, Theorem 3.5]{Pazy},
implies that $\partial_t u \in \mathcal{C}^{0,\beta} ((0,T),L^p
(\R^d))$ and $u \in \mathcal{C}^{0,\beta} ((0,T),H^{1,p} (\R^d))$
which is the desired regularity result.
\medskip

If $\alpha \in (1,2)$, we have $u \in L^q ((0,T),H^{\alpha-1,p}(\R^d))$
for all  $q \in
(1,\infty)$ and $p \in (p_\alpha,\infty)$, where
$p_\alpha=\bar{p}_\alpha=d/(\alpha-1)$.
 Equations~\eqref{Hbound} and \eqref{Riesz3} imply that
\begin{equation}\label{eq:reg1}  \nabla^{\alpha-1} G(u) \in
  L^q((0,T),L^p (\R^d)) 
\end{equation}
for all $p \in (p_\alpha,\infty)$ and $q \in (1,\infty)$.  In
particular, $\Psi (u) = |u| \nabla^{\alpha-1} G(u) \in L^q((0,T),L^p
(\R^d))$.  Hence, the maximal regularity gives us
\begin{equation}\label{eq:reg2}
 u \in L^q((0,T),H^{1,p} (\R^d))
\end{equation}
for all $p \in (p_\alpha,\infty)$ and $q \in (1,\infty)$. We now write
once again $\nabla \cdot \Psi (u) = f_1+ f_2$ with
\[f_1 = \sgn u\, \nabla u \cdot \nabla^{\alpha-1} G(u), \ \ \ 
f_2 = |u| \nabla^{\alpha-1} (G'(u)\nabla u).\]
In view of \eqref{eq:reg1} and \eqref{eq:reg2}, we have 
\[f_1 \in L^q ((0,T),L^p (\R^d)).\] 
We now claim that 
\[ f_2 \in L^q((0,T),H^{-1+(2-\alpha),p} (\R^d)). \]
Indeed, 
\[ G'(u) \nabla u \in L^q ((0,T),L^p(\R^d)). \]
This implies that 
\[ \nabla^{\alpha-1} (G'(u) \nabla u) \in L^q
((0,T),H^{-1+(2-\alpha),p}(\R^d))\]
which, in turn, implies the claim. Hence, 
\[\nabla \cdot \Psi (u) \in L^q ((0,T),H^{-1+(2-\alpha),p}(\R^d)) \]
for all $p \in (p_\alpha,\infty)$ and $q \in
(1,\infty)$.  Then, the maximal regularity  implies that
\[ u \in L^q ((0,T),H^{1+(2-\alpha),p}(\R^d))\]
for all for all $p \in (p_\alpha,\infty)$ and $q \in (1,\infty)$, thus
we see that the space regularity of $u$ is improved. More generally,
the same argument shows that if
\[u \in L^q ((0,T),H^{\beta,p}(\R^d)), \quad \text{ with } \beta \le
\alpha, \] then
\[u \in L^q ((0,T),H^{\beta+(2-\alpha),p}(\R^d)).\]
Now choose the least integer $k \ge 1$ such that
$\beta_k=1+k(2-\alpha) > \alpha$, and notice that $\beta_k<2$. Then
\[f_2 \in L^q((0,T),H^{\beta_k-\alpha,p}(\R^d)) \subset L^q((0,T),L^p(\R^d)).\]
Then the maximal regularity implies that $ \partial_t u \in
L^q((0,T),L^p(\R^d))$, which implies, as was in the case $\alpha \in
(0,1]$, that $u \in \mathcal{C}^{0,\beta} ((0,T),L^p (\R^d))$ for all
  $p \in (1,\infty)$ and $\beta \in (0,1)$. The previous reasoning in
  spaces of the form $L^q((0,T),Y)$ extends readily to spaces of the
  form $\mathcal{C}^{0,\beta}((0,T),Y)$. This yields the desired
  regularity result in the case $\alpha \in (1,2)$.

Moreover, it is known by \cite[Ch.~4,~Theorem~3.2]{Pazy} that mild
solutions of the equation $ \partial_t u = \delta \Delta u + \nabla
\cdot f $ with $f \in \mathcal{C}^{0,\beta} ((0,T),L^p (\R^d))$ are in
fact weak solutions, \textit{i.e.} they satisfy the equation in the
sense of distributions.  Under these regularity properties, the proof
of the mass conservation property \eqref{mass:cons} is completely
standard.
\end{proof}

\subsection*{Convexity inequalities}

First, we show a simple but useful technical result involving monotone functions and the fractional Laplacian.
\begin{lemma}\label{lem:hg}
Let $\alpha\in (0,2]$. Assume that $g,\, h\in \mathcal{C}^1[0,\infty)$ are strictly increasing functions. 
Then, for every nonnegative $v\in \mathcal{C}^\infty_c(\R^d)$ we have
\[\int h(v) (-\Delta )^{\frac{\alpha}{2}}g(v) \dx\ge 0.\]
\end{lemma}
\begin{proof}
Notice that for $\alpha=2$ this lemma is obviously true, which one
checks integrating by parts.  Now, let $\alpha\in(0,2)$.  Since
$(-\Delta)^{\frac{\alpha}{2}} C=0$ for every constant $C\in\R$, we can
assume that $g(0)=0$. In the same way, we can assume that $h(0)=0$,
because $\int (-\Delta )^{\frac{\alpha}{2}}w \dx= 0$ for every $w\in
\mathcal{C}^\infty_c(\R^d)$.  Defining $w=g(v)$, it suffices to show that
\[\int h(g^{-1}(w)) (-\Delta )^{\frac{\alpha}{2}}w \dx\ge 0\]
 for all $w\in \mathcal{C}^\infty_c(\R^d)$ such that $w\geq 0$.  To do it,
 notice that $f\in \mathcal{C}^2[0,\infty)$ defined via the relation
   $f(s)=\int_0^s h(g^{-1}(\tau))\;{\rm d}\tau$ for $s\geq 0$ is
   convex (it suffices to check that $f''(s)\geq 0$). Hence, using
    the pointwise inequality
\begin{equation}\label{in:convex}
(-\Delta)^{\frac\alpha2} g(v)\leq g'(v) (-\Delta)^{\frac\alpha2} v,
\end{equation} 
(see, {\it e.g.}, \cite{CC04}, \cite[Lemma 1]{DI06}) 
 we obtain $ \int h(g^{-1}(w))
   (-\Delta )^{\frac{\alpha}{2}}w \dx\ge \int (-\Delta
   )^{\frac{\alpha}{2}}f(w) \dx=0. $
\end{proof}
Next, we formulate a crucial  technical tool used in the
derivation of various integral estimates for solutions of the
regularized problem \eqref{reg}.

\begin{proposition}[Convexity inequalities]\label{prop:useful}
Consider a $\mathcal{C}^2$ function $\varphi: \R \to \R^+$  such
  that, for all $r\in \R$, $r \neq 0$, $\varphi''(r)>0$, and
\begin{equation}\label{growth:phi}
 \varphi(r)+|\varphi'(r)|+\varphi''(r) \le C (|r|^{M_1}+|r|^{M_2})
\end{equation}
for some constant $C>0$ and $M_1,M_2 \in [1,\infty)$. Then for all 
$0\le s<t\le T$, the function $u$ given by
  Proposition~\ref{prop:loc-time-ex} satisfies 
\begin{equation}\label{convex}
\int \varphi(u(t,x)) \dx + \int_s^t \int \psi (u(\tau,x))
(-\Delta)^{\frac\alpha2} G(u(\tau,x)) \dx\, {\rm d} \tau 
+ \delta
\int_s^t \int \varphi''(u) |\nabla u|^2\dx\, {\rm d} \tau \le \int \varphi (u(s,x)) \dx, 
\end{equation}
where $\psi(r) =  |r| \varphi'(r) -  \varphi (r)\, \sgn r$.
\end{proposition}
The proof of Proposition \ref{prop:useful} is more or less classical,
and we recall it in Appendix for the sake of completeness.

\begin{remark}
Remark that $\psi'(r)=|r|\varphi''(r)>0$, hence the function $\psi$ is
increasing. Since $G$ is also increasing, the result stated in
Lemma \ref{lem:hg} can be applied to show that the
first dissipation term
\[ \int_s^t \int \psi (u(\tau,x))
(-\Delta)^{\frac\alpha2} G(u(\tau,x)) \dx\, {\rm d} \tau \] 
is nonnegative.  The fact that this quantity is finite is a part of the
result stated in Proposition \ref{prop:useful}.  Moreover,
\cite[Theorem~2.2]{LS} implies that $g_\varphi (G (u)) \in L^2
((0,T),H^{\frac{\alpha}{2},2}(\R^d))$ for a function $g_\varphi$
constructed from $\varphi$, see \cite{LS} for the detailed
presentation. The special case $\varphi(r)=|r|^p$ is treated below.
\end{remark}
\begin{remark}
The convexity of $\varphi$ also implies that the second dissipative
term in \eqref{convex} is nonnegative. Hence, Proposition
\ref{prop:useful} implies that $\int \varphi (u(t,x)) \dx$ decreases
along the flow of the regularized equation \eqref{reg}.
\end{remark}

\begin{corollary}[Estimates of the $L^p$-norms]\label{cor:lp-norms}
For all $p \in (1,\infty)$ and $0 < s < t$,
\begin{equation} \label{Lp*-int}
 \int |u(t)|^p\dx + (p-1)\int_s^t \int |u|^{p-1} u (-\Delta)^{\frac\alpha2} G(u)\dx \dta 
+ \delta p (p-1) \int_s^t \int |u|^{p-2}|\nabla u|^2 \dta \le  \int |u(s)|^p\dx .
\end{equation}
In particular, for $p \ge \bar{p}_\alpha$ (see Corollary~\ref{cor:reg}), 
\begin{equation} \label{Lp*}
 \frac{\rm d}{\dt}\int |u|^p\dx 
\le - (p-1)\int |u|^{p-1} u (-\Delta)^{\frac\alpha2} G(u)\dx -
\delta p (p-1) \int |u|^{p-2}|\nabla u|^2.
\end{equation}
 Thus,  for all $p \in [1,\infty]$, the norm $\|u(t)\|_p$ decreases as
  $t$ increases.
\end{corollary}
\begin{proof}
If $p \ge 3$, we can apply Proposition~\ref{prop:useful} with $0\le
s<t\le T$, $\varphi (r) = |r|^p$; indeed, in this case, $\varphi$ is
a $\mathcal{C}^2$-function and satisfies the growth assumption 
\eqref{growth:phi}
with $\langle
M_1,M_2\rangle=\langle p,p-2\rangle$. Next, since $u \in
\mathcal{C}^1((0,T),L^p (\R^d))$, we obtain \eqref{Lp*} from the
inequality in Proposition \ref{prop:useful} by a direct computation.
We leave the details to the reader.

If $p \in (1,2)$, we consider, for each $\eta >0$, the function
  $\varphi_\eta$ such that $\varphi_\eta (0)=\varphi_\eta'(0)=0$ and 
\[\varphi_\eta''(r)
  = p(p-1) ((r^2+\eta^2)^{\frac{p}{2}-1} - \eta^{p-2}).\] In
  particular, the function $\varphi_\eta$ satisfies the assumptions of
  Proposition~\ref{prop:useful}, hence, we have
\begin{multline*}
\int \varphi_\eta(u(t,x)) \dx + \int_s^t \int \psi_\eta (u(s,x))
(-\Delta)^{\frac\alpha2} G(u(s,x)) \dx\, {\rm d} s \\
+ \delta \int_s^t \int \varphi''_\eta(u(s,x)) |\nabla u (s,x)|^2 \dx \,
{\rm d} s \le \int \varphi_\eta
(u(s,x)) \dx,
\end{multline*}
with $\psi_\eta'(r)=|r|\varphi_\eta''(r)$, $\psi_\eta(0)=0$. 
Letting now $\eta \to 0$ and using the Fatou lemma  yields 
the integral formulation of inequality \eqref{Lp*}. 

Thus, we just proved that $\|u(t)\|_p\leq \|u_0\|_p$ for all $t>0$ and
$p\in(1,\infty)$.  By computing  the limits as $p\to \infty$ and $p \to 1$,
the bounds $\|u(t)\|_1 \le \|u_0\|_1$ and $\|u(t)\|_\infty\le
\|u_0\|_\infty$ are also obtained.
\end{proof}
We can now complete the proof of  Theorem~\ref{thm:reg}. 
\begin{proof}[Proof of Theorem~\ref{thm:reg}]
In view of Proposition~\ref{prop:loc-time-ex} and
Corollary~\ref{cor:reg}, it remains to prove that solutions
are  nonnegative if   initial data are so, and that solutions are
global in time. 

The positivity property is derived immediately in a usual way from the
conservation of mass property \eqref{thm:mass} and the monotonicity of
the $L^1$-norm.  Indeed, \eqref{thm:mass} yields
\begin{equation*} 
\int u (T,x) \dx= \int u_+(T,x) \dx - \int u_-(T,x) \dx= \int ((u_0)_+
-(u_0)_-) (x)\dx
\end{equation*}
and
\begin{equation*} 
\int |u (T,x)| \dx= \int u_+(T,x) \dx + \int u_-(T,x) \dx\leq \int ((u_0)_+
+(u_0)_-) (x)\dx
\end{equation*}
(here, as usual, $u_+=\max\{0,u\}$ and $u_-=\max\{0,-u\}$).
These inequalities imply
$\int u_-(T,x) \dx\le \int (u_0)_-(x) \dx$, and,
in particular, the assumption $(u_0)_-=0$  gives  us $u\ge 0$ a.e.
\medskip

As far as the global existence of solutions is concerned, we argue as
follows.

For $\alpha\in(0,1]$, the time interval, where a solution is
constructed via the Banach fixed point theorem, depends only on
$\|u_0\|_1+ \|u_0\|_\infty$, and this norm of the solution does not
increase. Hence, we extend $u=u(t,\cdot)$ to the whole half-line
$[0,\infty)$, step-by-step.

For $\alpha\in(1,2)$, the Duhamel formula \eqref{Duh} and inequality
\eqref{Riesz3} with $v=0$ yield
\[ \|u(t)\|_{H^{\alpha-1,p}} \le \|u_0\|_{H^{\alpha-1,p}} +C_1(\|u_0\|_{p^\ast})\int_0^t 
(t-s)^{-\frac{\alpha}{2}}\|u(s)\|_{H^{\alpha-1,p}} \ds.\] 
Due to the singular Gronwall lemma, see, {\it e.g.}, \cite[Ch. 5,
  Lemma 6.7]{Pazy}, we deduce that the norm
$\|u(t)\|_{H^{\alpha-1,p}}$ cannot explode in finite time. This shows
that local-in-time solutions of the regularized equation \eqref{reg}
can be also continued to global-in-time ones.
\end{proof}

\section{Hypercontractivity and compactness estimates}

\subsection*{Hypercontractivity estimates}

We now turn to prove certain $L^1\mapsto L^p$  estimates for solutions of problem \eqref{reg}.

\begin{theorem}[$L^p$-decay of solutions to the regularized problem] 
\label{thm:Lp:reg}
Let $u=u(t,x)$ be a solution to the regularized problem \eqref{reg}
constructed in Theorem \ref{thm:reg}.  There exists a~constant
$C=C(d,\alpha,m)>0$ such that for all $\e>0$, $\delta>0$ and $p \in
[1,\infty]$,
\begin{equation}\label{Lpdec:reg}
\|u(t)\|_p\le C \|u_0\|_1^{\frac{d(m-1)/{p} +{\alpha}}{d(m-1)+\alpha}}
t^{-\frac{d}{d(m-1)+\alpha}\big(1-\frac{1}{p}\big)}
\end{equation}
for all $t>0$. 
\end{theorem}
\begin{proof}  We first remark that it is enough to prove the
decay estimate \eqref{Lpdec:reg} for large $p$'s, since the general
result follows by the interpolation of the $L^p$-norms combined with
the estimate $\|u(t)\|_1\le \|u_0\|_1$ from
Corollary~\ref{cor:lp-norms}. This is the reason why we will now prove
\eqref{Lpdec:reg} for $p \ge \max\{m-1,1,\bar{p}_\alpha\}=p_m$ (see
Corollary~\ref{cor:reg} for a definition of $\bar{p}_\alpha$).

We also remark that we can assume that $M=\|u_0\|_1=1$ by rescaling
the solution $u$ in the following way. First, we consider the function
$\widetilde u(t)=\frac{1}{M}u\left(\frac{t}{M^{m-1}}\right)$ which
satisfies equation~\eqref{reg} with suitably rescaled parameters:
$\tilde \delta = \frac{\delta}{M^{m-1}}$ and (if applicable) $\tilde\e
= \frac{\e}M$. Scaling back, we recover the desired inequality
\eqref{Lpdec:reg}.

We first prove the result when Corollary~\ref{cor:lp-norms} holds true
with $G(r)=|r|^{m-2}r$ and in the differential form~\eqref{Lp*}.  This
is the case when $m=2$ or $m \ge 3$. In the case $m \in (1,3)$,
Corollary~\ref{cor:lp-norms} holds true only in the integral
sense~\eqref{Lp*-int} and for a regularized function $G$. We will see
below how to pass to the limit as the regularization parameter $\eps$
goes to $0$ and get \eqref{Lp*-int} with $G(r)=|r|^{m-2}r$. For
expository reasons, we prefer to present the proof when we indeed have
a differential inequality, and then to explain how to adapt it if only
an integral version of it is available.
\bigskip

The proof on Theorem~\ref{thm:Lp:reg} in the cases $m \ge 3$ and $m=2$
is split into two steps: first, we show inequalities \eqref{Lpdec:reg}
with non-optimal constants $C$ which blow up for $p=\infty$; then, we
improve those constants by an iteration method.

{\it Decay estimates with optimal exponents and nonoptimal
  constants.}  Our computation consists in getting the following
differential inequality for $p \in (p_m,\infty)$.
\begin{lemma}\label{lem:diffineq}
There exists a constant $K=K(p,m)>0$ independent of $\e>0$ and
$\delta>0$ such that $K$ and $K^{-1}$ are bounded as $p \to \infty$
and
\begin{equation}\label{diff:ineq}
  \frac{\rm d}{\dt}\|u\|_p^p\le - K\|u\|_p^a,
 \end{equation}
with   $a$  defined in \eqref{a-b}. 
\end{lemma}
\begin{proof}
 We get from Corollary~\ref{cor:lp-norms}
\begin{align} 
  \frac{\rm d}{\dt}\int |u|^p\dx 
  &\leq -(p-1)\int u|u|^{p-1}(-\Delta)^{\frac\alpha2}(u |u|^{m-2})\dx
\label{Lp}\\
&\le -\frac{4p(p-1)(m-1)}{(p+m-1)^2}
\left\|\nabla^{\frac\alpha2}\left(|u|^{\frac{p+m-1}2}\right)
\right\|^2_2, \nonumber 
\end{align}
after applying the Stroock--Varopoulos inequality \eqref{SV} with
$w=u|u|^{m-2}$ and $q=\frac{p}{m-1}+1$. We use next the
Gagliardo--Nirenberg inequality from Lemma~\ref{GN} combined with  $\|u(t)\|_1 \le \|u_0\|_1 =1$ to estimate the
right-hand side of the above inequality. Thus, we get the differential
inequality \eqref{diff:ineq} with the constant
\begin{equation}
K=\frac{4(m-1)p(p-1)}{C_N(p+m-1)^2}\label{K-opt}.
\end{equation}  
\end{proof}
With the differential inequality \eqref{diff:ineq} in hand,  a
direct computation shows that every nonnegative solution of the
inequality $ \frac{\rm d}{\dt}f(t)\le -K f(t)^{\frac{a}p}
$ has to satisfy the algebraic decay 
\[f(t)\le
\left(K\left(\frac{a}p-1\right)\, t\right)^{-\frac{1}{\frac{a}p-1}}.\]
We recognize \eqref{Lpdec:reg} for $p> p_m$ with the constant
$C_p=\left(K\left(\frac{a}{p}-1\right)\right)^{-\frac{1}{a-p}}$.
Here, let us notice that the constants $C_p=C(d,\alpha,m,p)$ which are
obtained at this stage of the proof blow up as $p\to\infty$. Thus, we
cannot get the $L^\infty$-bound directly in this way.
\bigskip

\noindent {\it Recurrence step.}  To improve the constant $C_p$ and to
handle the limit case $p=\infty$, we apply a variation on the
Moser--Alikakos method of estimating the $L^p$-norms with $p=2^n$
recursively, see, {\em e.g.}, \cite{A0} and \cite[Lemma 3.1]{KMX}.
The starting point is the already obtained estimate for $p=2^k\ge m-1$
with the least integer $k$. 
\begin{lemma}\label{lem:recursive}
For each $n\ge k$, the following estimate
\begin{equation}\label{rec}
\|u(t)\|_{2^n}\le \kappa_nt^{-\mu_n} \ \ \ {\rm for\ all}\ \  t>0, 
\end{equation}
holds true with $\mu_n=\frac{1-2^{-n}}{\frac{\alpha}{d}+m-1}$ and with
a positive $\kappa_n$ satisfying the recursive estimate
\begin{equation}\label{rec-kappa}
  \kappa_{n+1}\le \bigg[
    \frac{2^n\left(\frac{2\alpha}d + \frac{m-1}{2^n}\right)\mu_n
      +1}{K_n \big(\frac{\alpha}{d}+\frac{m-1}{2^n}\big)}
\kappa_n^{2^n\big(\frac{2\alpha}{d}+\frac{m-1}{2^n}\big)}\bigg]^{2^{-n-1}
\frac{1}{\frac{\alpha}{d}+\frac{m-1}{2^n}}},
\end{equation}
where $K_n$ is given by formula \eqref{K-opt}  with $p = 2^{n+1}$. 
\end{lemma}
Thus, having this estimate we see that
$\limsup_{n\to\infty}\kappa_n<\infty$ (irrespective of the value of
$\kappa_k$ at the beginning of the recurrence), essentially since
$\sum_{n=k}^\infty n2^{-n}<\infty$,  see Appendix~\ref{app:recursive}
for details. Recall that the constants $K=K_n$ in the preliminary
estimates have been such that $K_n$ and $K_n^{-1}$ were bounded
uniformly when $p=2^n\to\infty$.
\begin{proof}[Proof of Lemma~\ref{lem:recursive}]
  We combine \eqref{Lp} with the Nash inequality \eqref{Nash-r} and
  two H\"older inequalities with $2^{n+1}<r=2^{n+1}+m-1\le
  2^{n+2}$
\[\|u\|_{2^{n+1}}\le \|u\|_r^\gamma\|u\|_{2^n}^{1-\gamma} \qquad \text{with}
\quad \frac{1}{\gamma}=2-\frac{2^{n+1}}{r},\] 
and
\[ \|u\|_{\frac{r}2}\le
\|u\|_{2^{n+1}}^\delta\|u\|_{2^n}^{1-\delta}\qquad \text{with} 
\quad \delta=2-\frac{2^{n+2}}{r}.\] 
As the result, we have
\[\frac{\rm d}{{\rm d}t}\|u\|_{2^{n+1}}^{2^{n+1}}\le
-K\|u\|_{2^{n+1}}^{2^{n+1}\big(1+\frac{\alpha}{d}+\frac{m-1}{2^n}\big)}
\|u\|_{2^n}^{-2^n\big(\frac{2\alpha}{d}+\frac{m-1}{2^n}\big)} \]
 with some $K$ as in \eqref{diff:ineq} corresponding to $p=2^{n+1}$.
Next we estimate the $L^{2^n}$-norm of $u(t)$ using \eqref{rec}, and
arrive to the differential inequality for
$f(t)=\|u(t)\|_{2^{n+1}}^{2^{n+1}}$ of the form
\begin{equation}\label{diff:ineq2} 
\frac{\rm d}{{\rm d}t}\big(f(t)^{-\frac{\alpha}{d}-\frac{m-1}{2^n}}\big)\ge
\frac{K\big(\frac{\alpha}{d}+\frac{m-1}{2^n}\big)}{ \kappa_n^{2^n\big(\frac{2\alpha}{d}+\frac{m-1}{2^n}\big)}}
  t^{\mu_n2^n\big(  \frac{2\alpha}{d}+\frac{m-1}{2^n}\big)} .
\end{equation}
 Finally, we integrate this inequality on $[0,t]$ and take a
suitable negative power to arrive to \eqref{rec-kappa}.   
\end{proof}

 Interpolation between $p=2^n$ and $p=2^{n+1}$ and the passage to the
 limit $p\to\infty$ finish  the proof of the hypercontractivity
 estimates in the cases $m \ge 3$ and $m=2$.
\medskip

To deal with the case $m \in (1,3)$, instead of the differential
inequality from Corollary \ref{cor:lp-norms}, we use its integral
counterpart \eqref{Lp*-int}.  In particular, inequalities \eqref{diff:ineq} and
\eqref{diff:ineq2} have their integral counterparts as well (see
inequality \eqref{estim:compact}, below). Hence, the previous proof
works in this case by applying the following lemma with $g(t)=t$ and
$g(t)=Ct^\nu$ for some well chosen positive constants $\nu,\,C$,
successively.
\begin{lemma}\label{lem:integral}
Consider functions  $f:[0,T] \to (0,\infty)$ nonincreasing, and $g:[0,T] \to
(0,\infty)$  increasing, smooth, and $g(0)=0$. Assume that for
a.e. $t \in (0,T)$, $s <t$,
\[ f(t) + K \int_s^t f(\tau)^{\gamma+1} g'(\tau) \dta \le f(s). \]
 Then for a.e. $t \in [0,T]$, we have 
\[ f(t) \le (K \gamma g(t))^{-\frac1\gamma}.\]
\end{lemma}
The proof of this lemma is given in Appendix~C for the readers'
convenience. Now, the proof of Theorem~\ref{thm:Lp:reg} is complete.
\end{proof}

\subsection*{Compactness estimates}

Now, we prove estimates which will allow us to pass to the limit as
$\eps \to 0$ and $\delta \to 0$, successively, in the regularized
problem \eqref{reg}.  We are going to use the fact that the
approximating functions $G=G_\varepsilon$ are $\gamma$-H\"older
continuous with $\gamma = \min \{m-1,1\}$ on the interval $\big[0,
  \|u_0\|_\infty\big]$, uniformly in $\varepsilon \in [0,1]$.  These
estimates are formulated in the following technical lemmas.
\begin{lemma}\label{lem:tek1}
Assume that $G$ is $\gamma$-H\"older continuous with $\gamma = \min \{m-1,1\}$.
For every $\alpha \in (0,1]$ and $m > 1+ \frac{1-\alpha}d\in(1,2)$,
  we have
\begin{equation}\label{estim:tek1}
\|\nabla^{\alpha-1} G (u) \|_q \le C \|u\|_{\gamma p}^\gamma
\end{equation}
with
\[q \ge (m-1-(1-\alpha)/d)^{-1} \text{ and } p = (1/q+(1-\alpha)/d)^{-1}.\]
\end{lemma}
\begin{proof}
Here, it suffices to combine the estimate on the Riesz potential with
\( \|G (u)\|_p \le C \|u\|_{\gamma p}^\gamma.\) The choice of $q$ (or
equivalently, the restriction on $m$) ensures that $\gamma p \ge 1$.
\end{proof}

The compactness of a sequence of solutions to the regularized problem
\eqref{reg} in the case $\alpha \in (1,2)$ is a consequence of the
following estimate.
\begin{lemma}\label{lem:onemore}
Consider a $\gamma$-H\"older continuous function $G: \R \to \R$. Then,
for every $\alpha \in (1,2)$ such that $(2-\gamma) \alpha <2$, we have
\[ \|\nabla^{\alpha-1} G(u)\|_p \le C \|u\|_{H^{\frac{\alpha}{2},2}} \]
with 
\[ \frac{\alpha-1}{\gamma} - \frac{d}p = \frac{\alpha}2 - \frac{d}2.\]
\end{lemma}
\begin{remark} 
Note that the assumptions of Lemma \ref{lem:onemore} 
 ensure that $p>2$. 
\end{remark}
\begin{proof}[Proof of Lemma \ref{lem:onemore}.]
We use
successively the fact that $H^{s,p} (\R^d) = F^s_{p,2} (\R^d)$
\cite[p.14]{rs96}, the characterization of Triebel-Lizorkin spaces
with difference quotients \cite[p.41]{rs96}, and finally known embedding
theorems for Besov and Triebel-Lizorkin spaces \cite[p.31]{rs96} in order to derive
\[ \|\nabla^{\alpha-1} G(u)\|_p \le C \| G(u)\|_{F^{\alpha-1}_{p,2}} \le C [G]_\gamma \| u
\|_{F^{\frac{\alpha-1}\gamma}_{p,2\gamma}} \le C \|u\|_{H^{\frac{\alpha}{2},2}} ,\]
where 
 $[G]_\gamma := \sup_{r \neq s} {|G
  (r)-G (s)|}/{|r-s|^\gamma} $
and $p>2$ is chosen so that the so-called differential dimension is
constant; this yields the condition appearing in the statement of the
lemma. The proof is now complete.
\end{proof}

\begin{lemma}\label{lem:tek2}
For $\alpha \in (1,2)$ and $(3-m) \alpha <2$, there exists $p \in
(1,\infty)$ such that 
\begin{equation}\label{estim:tek2}
\| \nabla^{\alpha-1} G(u)\|_p \le C 
\|u\|_{H^{1,2}(\R^d)} 
\end{equation}
where  $C$ depends on the $\gamma$-H\"older seminorm of $G$ in
$(0,\|u\|_\infty)$ {\rm (}with $\gamma = \min \{m-1,1\}${\rm )}. 
\end{lemma}
\begin{proof}
Apply Lemma~\ref{lem:onemore} with $\gamma = m-1$. 
\end{proof}
\begin{lemma}\label{lem:tek3}
If $\alpha \in (1,2)$ and $m > \alpha$, there exists $p \in
(1,\infty)$ and $r \in (1,\infty)$ {\rm (}and $r > m-1${\rm )} such that
\begin{equation}\label{estim:tek3}
\| \nabla^{\alpha-1}  |u|^{m-1} \sgn u\|_p \le 
C \| |u|^{\frac{r+m-1}2} \sgn u\|_{H^{\frac\alpha2,2}(\R^d)}.
\end{equation}
\end{lemma}
\begin{proof}
Consider $v = |u|^{\frac{r+m-1}2} \sgn u$. Then estimate \eqref{estim:tek3} is
equivalent to the following one
\[ \| \nabla^{\alpha-1} |v|^{\frac{2(m-1)}{r+m-1}} \sgn v \|_p \le C
\| v\|_{H^{\frac{\alpha}2,2}(\R^d)}.\] We then apply
Lemma~\ref{lem:onemore} with $\gamma = \frac{2(m-1)}{r+m-1}$ and get
the desired result if $\gamma > 2(1-\frac1\alpha)$, or equivalently,
\[ r < \frac{m-1}{\alpha-1}.\]
Here, in order to find $r \in\big(1,\frac{m-1}{\alpha-1}\big)$, the condition $m
> \alpha$ is needed. The proof is now complete.
\end{proof}

\section{Proof of Theorem~\ref{thm:hyper}}

This section is devoted to the proof of our main result on the
existence of solutions to problem \eqref{eq:main}--\eqref{eq:ic}
satisfying decay estimates \eqref{u:decay}.

\begin{proof}[Proof of Theorem~\ref{thm:hyper}] We first consider very
regular initial data, \textit{i.e.} we assume that $u_0$ satisfies
\eqref{icc}.  Then, this condition is relaxed by considering initial
data that are merely integrable.

The proof proceeds in three steps: passage to the limit with the
parameter of the regularization of the nonlinearity, then with the
parameter of the parabolic regularization, and finally --- stability
with respect to  initial data.

\subsection*{Passage to the limit as $\eps \to 0$.}
Consider $u_0$ satisfying \eqref{icc}. From Theorem~\ref{thm:reg}, we
have a sequence of solutions $u_\e$ of \eqref{reg} for $G_\eps$
defined in such a way that there exists $C>0$ such that for all $\e \in (0,1]$,
\[ [G_\e]_\gamma := \sup_{r \neq s} \frac{|G_\e
  (r)-G_\e (s)|}{|r-s|^\gamma} \le C\] where $\gamma = \min
\{m-1,1\}$. Thanks to Corollary~\ref{cor:lp-norms}, there  exists also a
constant $C>0$ (depending on $\delta>0$) such that for all $\e \in (0,1]$,
\begin{equation}
\|u_\e\|_{L^\infty ((0,T) \times \R^d)} \le C, \label{estim:linftybis}  \qquad
\|u_\e \|_{L^2((0,T),H^{1,2}(\R^d))} \le C. 
\end{equation}
Hence, we can construct a sequence $(\eps_n)_n$ such that 
\[ u_n \rightharpoonup u \text{ in } L^2 ((0,T),H^{1,2} (\R^d))\]
where $u_n$ denotes $u_{\e_n}$. 
Moreover, for all $R>0$, the embedding 
$ H^{1,2} (B_R) \subset L^2(B_R) $
is dense and compact. Using \eqref{estim:linftybis} we also obtain
\[ \lim_{\text{meas}(E)\to 0, E \subset [0,T]} \int_E \int_{B_R} |u_n
(t,x)|^2 \dt \dx =0. \] Hence, we infer from \cite{RT-AML} (which
contains an optimal result on the compactness for Hilbert space valued
vector functions) that, up to a subsequence, for every $R>0$,
\[ u_n \to u \text{ in } L^2 ((0,T)\times B_R).\]
Moreover, passing again to a subsequence if necessary, we can assume that for every $R>0$,
\begin{equation}
\nabla u_n \rightharpoonup \nabla u \text{ in } L^2 ((0,T)\times B_R),  \ \
u_n \to u  \text{ for a.e. } (t,x) \in Q_T, 
\end{equation}
which imply  for all $\varphi \in {\mathcal C}^\infty_c (\bar{Q}_T)$,
\begin{align*}
\iint_{Q_T} u_n \partial_t \varphi \dt \dx &\to \iint_{Q_T} u \partial_t
\varphi \dt \dx, \\
\iint_{Q_T} \nabla u_n \cdot \nabla \varphi \dt \dx &\to \iint_{Q_T} \nabla u
\cdot \nabla \varphi\dt \dx. 
\end{align*}
As far as the nonlinear term in equation \eqref{reg} is concerned, we have
\[ |u_n| \to |u| \text{ for a.e. } (t,x) \in Q_T.\]
Hence,  the dominated convergence
theorem combined with  \eqref{estim:linftybis} imply that 
\[ |u_n| \nabla \varphi \to |u| \nabla \varphi \text{ in }
L^2 ((0,T),L^{q'} (\R^d))\] for all $q' \in (1,\infty)$.  
Lemmas~\ref{lem:tek1} and \ref{lem:tek2} imply that
\begin{equation}\label{estim:nl}
 \nabla^{\alpha-1}G_{\e_n} (u_n) \text{ is bounded in } L^q(Q_T)
\end{equation}
for some $q \in (1,\infty)$.

We deduce from \eqref{estim:nl} that we can extract a weakly
converging subsequence of $(\nabla^{\alpha-1}G_{\e_n} (u_n))_n\subset
L^q(Q_T)$ which limit is denoted by $L$. Since
\[G_{\e_n} (u_n) \to |u|^{m-1} \sgn u \qquad \text{ a.e. in } Q_T,\]
we have 
\begin{equation}
\nabla^{\alpha-1}G_{\e_n} (u_n) \rightharpoonup L \text{ in } L^q (Q_T),\quad\text{where}\quad 
L=\nabla^{\alpha-1} (|u|^{m-1}\sgn u) \text{ in } \mathcal{D}'(Q_T).\nonumber
\end{equation}
In particular, $\nabla^{\alpha-1} (|u|^{m-1}\sgn u) \in L^q(Q_T)$ holds,  and
\[ \iint_{Q_T} \nabla^{\alpha-1} G_{\e_n} (u_n) \cdot |u_n|
\nabla \varphi \dt \dx \to \iint_{Q_T} \nabla^{\alpha-1} (|u|^{m-1} \sgn u)
\cdot |u|\nabla \varphi \dt \dx \] 
as $n \to \infty$. Hence, we obtain
a weak solution of \eqref{reg} with $G(r)=|r|^{m-1} \sgn r$ for $u_0$
satisfying \eqref{icc}.
\medskip

By the Fatou lemma, we derive from Corollary~\ref{cor:lp-norms}
and Proposition~\ref{prop:SV} the following estimates: for a.e. $t>s>0$,
$q \in [1,\infty]$, $p \in (1,\infty)$, uniformly in $\delta\in(0,1]$, 
\begin{align}
\label{estim:lp}
\| u(t)\|_q  \le \|u_0\|_q, \\
\label{estim:compact}
\|u(t)\|_p^p + 
\frac{4p(p-1)(m-1)}{(p+m-1)^2} \iint_{(s,t) \times \R^d} \left|\nabla^{\frac\alpha2}\left(
\sgn u\,|u|^{\frac{p+m-1}2} \right) \right|^2 \ds \dx  \le \|u (s)\|_p^p, \\
\label{estim:contractive}
\|u(t)\|_q\le C \|u_0\|_1^{\frac{d(m-1)/q +{\alpha}}{d(m-1)+\alpha}}
t^{-\frac{d}{d(m-1)+\alpha}\big(1-\frac{1}q\big)}, \\
\label{estim:grad}
2 \delta \iint_{Q_t} |\nabla u|^2 \ds \dx \le \|u_0\|_2^2.
\end{align}

Remark that now we can derive the hypercontractivity estimates in the
case $m \in (1,3)$ from inequality \eqref{estim:compact}.  We also
know that the solution $u$ is nonnegative if $u_0$ is so.

\subsection*{Passage to the limit as $\delta \to 0$} Let $u=u^\delta$
denote the solution constructed above. 
We consider 
\[v^\delta = |u|^{m} \sgn u. \]
Using \eqref{estim:compact} with $s=0$ and $p=m+1 >1$, we get that,
for all $T>0$,
\[ v^\delta \text{ is bounded in } L^2 ((0,T), H^{\frac\alpha2,2}
(\R^d)). \]
Hence, there exists a subsequence $\delta_n \to 0$ such that 
\[ v_n \rightharpoonup v \quad \text{ in } L^2 ((0,T), H^{\frac\alpha2,2}
(\R^d)),\]
where  $v_n$ denotes $v^{\delta_n}$. Moreover, inequality \eqref{estim:lp}
with $q = \infty$  implies that for $R>0$ 
\[ \lim_{\text{meas}(E)\to 0, E \subset [0,T]} \int_E \int_{B_R} |v_n
(t,x)|^2 \dt \dx =0. \]
Hence, we can use \cite{RT-AML} once more,   and conclude that 
for all $R>0$ 
\[ v_n \to v \text{ in } L^2 ((0,T)\times B_R).\]
Moreover, passing to a subsequence if necessary, we can assume that 
\[u_n \to u  \text{ for a.e. } (t,x) \in Q_T,\]
where $u_n$ denotes $u^{\delta_n}$. 
Now Lemmas~\ref{lem:tek1} and \ref{lem:tek3} imply that 
\begin{equation}\label{estim:nl+}
 \nabla^{\alpha-1} (|u_n|^{m-1} \sgn u_n) \text{ is bounded in } L^q(Q_T)
\end{equation}
for some $q >1$. Hence, we can pass to the limit in the nonlinear
term in the weak formulation of equation \eqref{reg}.

To complete this proof, notice that inequality \eqref{estim:grad} implies that 
\[ \delta_n \nabla u_n \to 0 \text{ in } L^2(Q_T).\]
In particular,
\[\delta_n \iint_{Q_T} \nabla u_n \cdot \nabla \varphi \dt \dx \to 0\]
as $n \to \infty$. We thus conclude that $u$ is a weak solution of
\eqref{eq:main}--\eqref{eq:ic}. 

The conservation of mass, the positivity property, the monotonicity of
$L^p$-norms, and the hypercontractivity estimates follow from
\eqref{estim:lp}--\eqref{estim:grad} in a standard way.

\subsection*{Stability with respect initial conditions}

Assume now that $u_0 \in L^1(\R^d) \cap L^\infty (\R^d)$. Consider an
approximating sequence $u_0^n$ satisfying \eqref{icc}. Then, the
sequences of $L^p$-norms are bounded and we can pass to the limit as
we did already 
when letting $\delta \to 0$. We also recover the expected
properties of the weak solution. 

Assume finally that $u_0$ is merely integrable. Then $T_n (u_0) \in
L^1(\R^d) \cap L^\infty (\R^d)$ where 
\[T_n(r) = \min\{\max\{r,-n\},n\}.\]
 Hence, there exists a weak solution $u_n$ of
\eqref{eq:main}. In view of the hypercontractivity estimates, we can
extract a~converging subsequence in $(\eta,T) \times \R^d$, with arbitrary
$\eta>0$, to a function $u^\eta$ in the following sense
\[ u_n \to u^\eta \text{ in } L^2 ((\eta,T) \times B(0,R)).\]
 Moreover,
the $L^1$-norm of $u_n$ is bounded in $(0,T) \times \R^d$. Then we have 
\begin{align*}
\left|\iint_{(0,\eta) \times \R^d} u_n \partial_t \vphi \dt \dx
\right| &\le C \eta, \\ 
\iint_{(\eta,T) \times \R^d} u_n \partial_t \vphi 
& \to \iint_{(\eta,T) \times \R^d} u^\eta \partial_t \vphi\\
\iint_{(\eta,T)\times \R^d} \nabla^{\alpha-1}
(|u_n|^{m-1} \sgn u_n) \cdot \nabla \vphi \dt \dx &\to
\iint_{(\eta,T)\times \R^d} \nabla^{\alpha-1} (|u^\eta|^{m-1} \sgn
u^\eta) \cdot \nabla \vphi \dt \dx
\end{align*}
as $n \to \infty$ for each $\eta>0$.
We can now conclude the
proof of Theorem~\ref{thm:hyper}
 through a  diagonal procedure.
\end{proof}

\appendix

\section{Proof of Proposition~\ref{prop:useful}}
\label{app:useful}

\begin{proof}[Proof of Proposition~\ref{prop:useful}]
We fix an arbitrary $p \ge p_\alpha=\frac{d}{\alpha-1}$. First, we
remark that
\[
\iint (\partial_t u) \phi \dt \dx+ \int |u| \nabla^{\alpha-1} G (u)
 \cdot \nabla \phi \dt \dx + \delta \iint \nabla u \cdot \nabla \phi
\dt \dx= 0
\]
for all $\phi$ compactly supported in time and in
$L^1((0,T),H^{1,p'}(\R^d))$. Recall that we have $\partial_t u \in
\mathcal{C}[0,T], L^p (\R^d))$ for $p \ge p_\alpha =
\frac{d}{\alpha-1}$ and $\Psi(u)=|u| \nabla^{\alpha-1} G (u) \in L^\infty
((0,T),L^p (\R^d))$. To justify the previous equality, it is enough to
mollify the function $\phi$ in time and space, and remark that the
mollified function $\phi_\eta$ satisfies (for $\eta$ small enough)
$\phi_\eta (0,x)=\phi_\eta (T,x)=0$, $u \phi_\eta \in
\mathcal{C}^1((0,T),L^1(\R^d))$ and
\[ \iint u (\partial_t \phi_\eta) \dt \dx = - \iint (\partial_t u) \phi_\eta \dt \dx.\]
Letting $\eta \to 0$ and using the regularity of $\Psi(u)$ yields the
desired result.

Next, consider $\phi (\tau,x)= \varphi'(u(\tau,x)) \Theta_\eta(\tau)$,
 where $\Theta_\eta$ is truncation function in time of
$[s,t]$. Remark that $\phi \in L^1((0,T),H^{1,p'}(\R^d))$; indeed,
$\varphi'(u) \in L^\infty ((0,T),L^{p'} (\R^d))$, and we also have
for all $w \in (H^{1,p} \cap L^1 \cap L^\infty) (\R^d)$
\[
\|\varphi''(w) \nabla w\|_{p'} \le \|\nabla w \|_p \|\varphi''(w)\|_q
\]
with $\frac1q=1-\frac2p$. 
As far as $\Theta_\eta$ is concerned, we choose it such that $\Theta_\eta
(s)=0$ and $\Theta_\eta'(\tau) = \rho_\eta (\tau-s-\eta) - \rho_\eta (\tau-t+\eta)$
where $\rho_\eta$ is an even mollifier supported in $[-\eta,\eta]$.
Now we can write
\begin{align*}
 \iint (\partial_t u) \varphi'(u) \Theta_\eta \dta \dx
&= \iint \partial_t (\varphi(u)) \Theta_\eta \dta \dx \\
& =  \iint \varphi(u) (\rho_\eta (\tau-t+\eta) -\rho_\eta (\tau-s-\eta)) \dta \dx.
\end{align*}
Hence,
\[
\iint (\partial_t u) \varphi'(u) \Theta_\eta \dta \dx \to \int \varphi (u(t,x))
\dx- \int \varphi (u(s,x))\dx .\] 
Moreover,
\[ |u| \nabla \phi (u)= |u| \varphi''(u) \nabla u = \nabla (\psi (u)) \quad
\text{ in } L^p(\R^d)\] 
thanks to the Stampacchia theorem ($u \in
L^\infty(\R^d)$ hence $\varphi'$ and $\psi$ locally Lipschitz is enough). Hence
\begin{eqnarray*}
 \int_s^t \int |u|\nabla^{\alpha-1} G (u) \cdot
  \nabla \phi(u) \dta \dx&  =  &
\int_s^t \int \nabla^{\alpha-1} G (u) \cdot \nabla (\psi (u))\dx
\dta. 
\end{eqnarray*}
It remains to prove that 
\begin{equation}\label{ineq:fraclap} 
 \int_s^t \int \psi(u) (-\Delta)^{\frac\alpha2} (G (u))\dta \dx
\le 
\int_s^t \int \nabla^{\alpha-1} G (u) \cdot \nabla (\psi (u))\dx
\dta < \infty.
\end{equation}
The last inequality comes from the computations we made above. As far
as the second inequality is concerned, we use \eqref{frac:int} to write
\[\int_s^t \int \nabla^{\alpha-1} G (u) \cdot \nabla (\psi (u))\dx
\dta = \lim_{\eta \to 0} \int_s^t \iint (G (u)(\tau,x) -G (u)(\tau,y))
F_\eta (y-x) \cdot \nabla (\psi (u(\tau,x))) \dx \dy \dta\]
where $F_\eta$ is defined as follows
\[ F_\eta (z) = C_{d,\alpha} \frac{z}{\eta^{d+\alpha}+ |z|^{d+\alpha} }.\]
Through an integration by parts, we now get
\begin{multline*}
\int_s^t \iint (G(u)(\tau,x) -G (u)(\tau,y)) F_\eta (y-x) \cdot \nabla (\psi (u(\tau,x))) \dx \dy \dta \\
= -\int_s^t \iint \nabla G (u)(\tau,x) \cdot F_\eta
(y-x) \psi (u(\tau,x)) \dx \dy \dta \\
+ \int_s^t \iint ( u(\tau,x)-u(\tau,y)) (\nabla \cdot F_\eta)
(y-x) \psi (u(\tau,x)) \dx \dy \dta.
\end{multline*}
The first term on the right-hand side equals $0$ since $F_\eta$ is
odd. Moreover, we have 
\[ \nabla \cdot F_\eta (z) =  C_{d,\alpha}  \frac{\eta^{d+\alpha} +
  (d+\alpha+1)|z|^{d+\alpha}}{(\eta^{d+\alpha} +
  |z|^{d+\alpha})^2}.\] 
Hence, the Fatou lemma yields  \eqref{ineq:fraclap}.
 The proof of the proposition is now complete.
\end{proof}

\section{From \eqref{rec-kappa} to the boundedness of $\kappa_n$}
\label{app:recursive}

Consider $l_n = \log \kappa_n$ and write \eqref{rec-kappa} as follows
\[ l_{n+1} \le  a_n + b_n l_n \]
with 
\begin{align*}
a_n & = \frac{1}{2^{n+1}\left(\frac\alpha{d} +\frac{m-1}{2^n}\right)}
\log \left( \frac{2^n \left(\frac{2\alpha}d + \frac{m-1}{2^n}
  \right)+1}{K_n \left(\frac\alpha{d} + \frac{m-1}{2^n}\right)}
\right); \\
 b_n & = 1 - \frac{m-1}{\frac\alpha{d} 2^{n+1} + 2 (m-1)}.
\end{align*}
Remark next that 
\[
a_n  \le C \frac{n}{2^n} \quad \text{ and } \quad
b_n  \le 1 - \frac{C}{2^n}.
\]
In particular 
\[ \sum_{n \ge k} a_n <  \infty \quad \text{ and } \quad 
\prod_{n \ge k} b_n < \infty \]  
Using the fact that $b_n \le 1$, we get
\[ l_n \le \sum_{n \ge k} a_n + \left(\prod_{n \ge k} b_n \right)a_{k_0}.\]
Hence, $l_n$ does not blow up, and neither does $\kappa_n$. 

\section{Proof of Lemma~\ref{lem:integral}}

\begin{proof}[Proof of Lemma~\ref{lem:integral}]
First, we remark that we can reduce to the case $g(t)=t$ and $K
=1$ through a change of variables. 

The proof is simple if $f$ is smooth. If $f$ is not, extend $f$ by $0$
to $\R$ and consider a mollifier $\rho_\eps$. Then write for $t_1 <
t_2$,
\[ f(t_1-s) + K \int_{-\infty}^{t_1-s} f^{\gamma+1} (\tau) \dta \le
f(t_2-s).\]
Now integrate against $\rho_\eps (s)$ and use the Jensen inequality to
get 
\[f_\eps (t_1)+ K \int_{-\infty}^{t_1} f_\eps^{\gamma+1} (\tau) \dta
\le f_\eps (t_2). \] We are now reduced to the case $f=f_\eps$, which
is smooth.  Passing to the limit, the proof is now complete.
\end{proof}

\bibliographystyle{siam} \bibliography{fpm1}

\end{document}